\documentclass[leqno,final]{siamltex}
\pdfoutput=1

\usepackage{graphicx}
\usepackage{amsmath,amstext,amssymb}

\newtheorem{thm}{Theorem}[section]

\newtheorem{lem}{Lemma}[section]
\newtheorem{remark}{Remark}[section]

\newcommand{\eps}{\varepsilon}
\newcommand{\nab}{\nabla}
\newcommand{\p}{\partial}
\newcommand{\Ome}{\Omega}
\newcommand{\Del}{\Delta}

\newcommand{\be}{{\mathbf{e}}}
\newcommand{\bw}{{\mathbf{w}}}
\newcommand{\bv}{{\mathbf{v}}}
\newcommand{\bu}{{\mathbf{u}}}
\newcommand{\bR}{{\mathbb{R}}}
\newcommand{\cT}{{\mathcal{T}}}
\newcommand{\lt}{{L^2}}
\newcommand{\ho}{{H^1}}
\newcommand{\htw}{{H^2}}
\newcommand{\Div}{{\mbox{\rm div\,}}}
\newcommand{\vepsi}{{\varepsilon}}
\newcommand{\spa}{\hspace{1cm}}
\newcommand{\normd}[1]{\frac{\partial #1}{\partial \nu}}

\newcommand{\oalpha}{{\overline{\alpha}}}
\newcommand{\oxi}{{\overline{\xi}}}
\newcommand{\oomega}{{\overline{\omega}}}

\title{A modified characteristic finite element method for a fully nonlinear
formulation of the semigeostrophic flow equations\footnote{The work of both 
authors was partially supported by the NSF grants DMS-0410266 and DMS-0710831.}}

\author{
Xiaobing Feng\thanks{Department of Mathematics, The University of
Tennessee, Knoxville, TN 37996, U.S.A.  ({\tt xfeng@math.utk.edu}).}
\and 
Michael Neilan\thanks{Department of Mathematics, The University of
Tennessee, Knoxville, TN 37996, U.S.A.  ({\tt neilan@math.utk.edu}).}
}

\begin{document}

\maketitle


\setcounter{page}{1}

\large
\baselineskip 13.5pt
\begin{abstract}
This paper develops a fully discrete modified characteristic 
finite element method for a coupled system consisting of
the fully nonlinear Monge-Amp\'ere equation and a transport
equation.  The system is the Eulerian formulation in the dual
space for the B. J. Hoskins' semigeostrophic flow equations,
which are widely used in meteorology to model slowly varying flows
constrained by rotation and stratification. To overcome the
difficulty caused by the strong nonlinearity, we first 
formulate (at the differential level) a vanishing moment
approximation of the semigeostrophic flow equations, 
a methodology recently proposed by the authors \cite{Feng1,Feng2},
which involves approximating the fully nonlinear Monge-Amp\'ere equation 
by a family of fourth order quasilinear equations.
We then construct a fully discrete modified characteristic 
finite element method for the regularized problem.
It is shown that under certain mesh and time stepping
constraints, the proposed
numerical method converges with an optimal order rate of 
convergence. In particular, the obtained error bounds 
show explicit dependence on the regularization parameter $\vepsi$.
Numerical tests are also presented to validate the theoretical 
results and to gauge the efficiency of the proposed fully discrete 
modified characteristic finite element method.
\end{abstract}



\section{Introduction}\label{sec-1}
The semigeostrophic flow equations, which were derived by 
B. J. Hoskins \cite{hoskins75},
is used in meteorology to model slowly varying flows constrained by rotation 
and stratification. They can be considered as an approximation of the Euler 
equations and are thought to be an efficient model to describe front formation 
(cf. \cite{Loeper_06,Cullen_Douglas99}).  Under certain assumptions and in some 
appropriately chosen curve coordinates (called `dual space', see Section \ref{sec-2}), 
they can be formulated as the following coupled system consisting of the 
fully nonlinear Monge-Amp\'ere equation and the transport equation:
\begin{alignat}{2}\label{intro1}
\det(D^2\psi^*)&=\alpha\spa &&\text{in }\mathbb{R}^3\times (0,T],\\
\frac{\p \alpha}{\p t}+\Div(\mathbf{v}\alpha)&=0\spa &&\text{in }\mathbb{R}^3\times 
(0,T], \label{intro2}\\
\alpha(x,0)&=\alpha_0\spa &&\text{in }\mathbb{R}^3\times \{t=0\}, 
\label{intro3} \\
\nabla \psi^*&\subset \Omega, \label{intro3a}
\end{alignat}
and
\begin{align}\label{vdef}
\mathbf{v}=(\nab \psi^* - x)^\bot=(\psi^*_{x_2}-x_2,x_1-\psi^*_{x_1},0).
\end{align}
Here, $\Ome\subset \bR^3$ is a bounded domain, 
$\alpha$ is the density of a probability measure on $\mathbb{R}^3$, 
 and $\psi^*$ denotes the Legendre transform of a convex function $\psi$. 
For any $\bw=(w_1,w_2,w_3)$, $\bw^\bot:=(w_2,-w_1,0)$. We note 
that none of the variables $\alpha, \psi^*$, and $\bv$ in the system
is an original primitive variable appearing in the Euler equations.
However, all primitive variables can be conveniently recovered from these 
non-physical variables (see Section \ref{sec-2} for the details).  

In this paper, our goal is to numerically approximate the solution of
\eqref{intro1}--\eqref{vdef}.  By inspecting the above system, 
one easily observes that there are three clear difficulties for
achieving the goal. First, the equations are posed over an unbounded 
domain, which makes numerically solving the system infeasible. 
Second, the $\psi^*$-equation is the fully nonlinear Monge-Amp\'ere equation. 
Numerically, little progress has been made in approximating second order 
fully nonlinear PDEs such as the Monge-Amp\'ere equation. 
Third, equation \eqref{intro3a} imposes a nonstandard
constraint on the solution $\psi^*$, which often 
is called the second kind boundary condition for $\psi^*$ in the PDE community 
(cf. \cite{Benamou_98,Cullen_Douglas99}).

As a first step to approximate the solution of the above system, we must 
solve \eqref{intro1}--\eqref{intro3} over a finite domain, $U\subset
\mathbb{R}^3$, which then calls for the use of artificial boundary 
condition techniques. For the second difficulty, we recall that
a main obstacle is the fact that weak solutions (called viscosity 
solutions) for second order nonlinear PDEs are non-variational. This
poses a daunting challenge for Galerkin type numerical methods such
as finite element, spectral element, and discontinuous Galerkin
methods, which are all based on variational formulations of PDEs.
To overcome the above difficulty, recently we introduced a new approach 
in \cite{Feng1,Feng2,Feng3,Feng4,Neilan_08}, called
the vanishing moment method in order to approximate viscosity solutions of
fully nonlinear second order PDEs.  This approach gives rise a new
notion of weak solutions, called moment solutions, for fully
nonlinear second order PDEs. Furthermore, the vanishing moment
method is constructive, so practical and convergent numerical
methods can be developed based on the approach for computing 
viscosity solutions of fully nonlinear second order PDEs. The main idea
of the vanishing moment method is to approximate a fully nonlinear
second order PDE by a quasilinear higher order PDE.  In this paper,
we apply the methodology of the vanishing moment method, and
approximate \eqref{intro1}--\eqref{intro3} by the following
fourth order quasi-linear system:
\begin{alignat}{2}\label{intro4}
-\eps\Delta^2\psi^\eps+\det(D^2\psi^\eps)
&=\alpha^\eps\spa &&\text{in }U\times (0,T],\\ 
\frac{\p \alpha^\eps}{\p t}+\Div(\mathbf{v}^\eps\alpha^\eps)
&=0\spa &&\text{in }U\times (0,T], \label{intro5}\\
\alpha^\eps(x,0)&=\alpha_0(x)\spa &&\text{in }\bR^3\times \{t=0\}, \label{intro6}
\end{alignat}
where
\begin{align}\label{vepsdef}
\bv^\eps:=(\nab \psi^\vepsi-x)^\bot=(\psi^\eps_{x_2}-x_2,x_1-\psi^\eps_{x_1},0).
\end{align}
It is easy to see that \eqref{intro4}--\eqref{vepsdef} is underdetermined,
so extra constraints are required in order to ensure uniqueness.
To this end,  we impose the following boundary conditions and constraint
to the above system:
\begin{alignat}{2}\label{intro7}
\normd{\psi^\eps}&=0\spa &&\text{on }\partial U\times (0,T],\\
\normd{\Delta \psi^\eps}&=\eps\spa &&\text{on }\partial U\times (0,T],
\label{intro8}\\
\int_U \psi^\eps dx&=0\spa &&t\in (0,T], \label{intro9}
\end{alignat}
where $\nu$ denotes the unit outward normal to $\p U$. 
We remark that the choice of \eqref{intro8} intends to minimize the 
boundary layer due to the introduction of the singular perturbation 
term in \eqref{intro4} (see \cite{Feng1} for more discussions).  
Boundary condition \eqref{intro7} is used to 
minimize the ``reflection" due to the introduction of the 
finite computational domain $U$.  It can be regarded as a 
simple radiation boundary condition.  An additional consequence
of \eqref{intro7} is that it also effectively overcomes the third 
difficulty, which is caused by the nonstandard constraint 
\eqref{intro3a}, for solving system \eqref{intro1}--\eqref{vdef}. 
Clearly, \eqref{intro9} is purely a mathematical technique 
for selecting a unique function from a class of functions 
differing from each other by an additive constant.

The specific goal of this paper is to formulate and analyze
a modified characteristic finite element method for problem
\eqref{intro4}--\eqref{intro9}. The proposed method 
approximates the elliptic equation for $\psi^\vepsi$ by conforming 
finite element methods (cf. \cite{Ciarlet78}) 
and discretizes the transport equation for $\alpha^\vepsi$ by 
a modified characteristic method due to Douglas and Russell 
\cite{Douglas_Russell_82}. We are particularly interested in
obtaining error estimates that show explicit dependence on $\vepsi$ 
for the proposed numerical method.

The remainder of this paper is organized as follows.  In Section
\ref{sec-2}, we introduce the semigeostrophic flow equations and show
how they can be formulated as the Monge-Amp\'ere/transport system
\eqref{intro1}--\eqref{vdef}. In Section \ref{sec-3}, we apply
the methodology of the vanishing moment method to approximate
\eqref{intro1}--\eqref{vdef} via \eqref{intro4}--\eqref{intro9}, 
prove some properties of this approximation, and also state certain 
assumptions about this approximation.  We then formulate our modified 
characteristic finite element method to numerically
compute the solution of \eqref{intro4}--\eqref{intro9}.
Section \ref{sec-4} mirrors the analysis found in \cite{Feng4} 
where we analyze the numerical solution of
the Monge-Amp\'ere equation under small perturbations of the data.
Section \ref{sec-4} is of independent interests in itself, but the main 
results will prove to be crucial in the next section.  In Section
\ref{sec-5}, under certain mesh and time stepping constraints,
we establish optimal order error estimates for the proposed
modified characteristic finite element method. The main idea of 
the proof is to use the results of Section \ref{sec-4} and an 
inductive argument.  Finally, in Section \ref{sec-6}, we provide 
numerical tests to validate the theoretical results of the paper.

Standard space notation is adopted in this paper, we refer to
\cite{Brenner,Gilbarg_Trudinger01,Ciarlet78} for their exact
definitions. In particular, $(\cdot,\cdot)$ and $\langle\cdot,\cdot\rangle$
denote the $L^2$-inner products on $U$ and $\p U$, respectively. $C$ is 
used to denote a generic positive constant which is independent of
$\vepsi$ and mesh parameters $h$ and $\Delta t$.

\section{Derivation of the Monge-Amp\'ere/transport formulation for
the semigeostrophic flow equations}\label{sec-2}
For the reader's convenience and to provide necessary background,
we shall first give a concise derivation of the Hoskins' semigeostrophic
flow equations \cite{hoskins75} and then explain how the Hoskins' model
is reformulated as a coupled Monge-Amp\'ere/transport system. 
Although our derivation essentially follows those
of \cite{hoskins75, Cullen_Douglas99, Benamou_98}, we shall make an effort
to streamline the ideas and key steps in a way which we thought
should be more accessible to the numerical analysis community.

Let $\Ome\subset \bR^3$ denote a bounded domain of the {\em troposphere}
in the atmosphere. It is well known \cite{Majda} that if fluids are assumed 
to be incompressible, their dynamics in such a domain $\Ome$ are governed by 
the following incompressible Boussinesq equations which are a version of the 
incompressible Euler equations:
\begin{alignat}{2}\label{Euler1}
\frac{D\bu}{Dt} + \nab p  &=f \bu^{\bot} 
-\frac{\theta}{\theta_0} g\be_3 &&\qquad \text{in }\Omega\times (0,T],\\ 
\frac{D\theta}{Dt} &= 0 &&\qquad \text{in }\Omega\times (0,T], \label{Euler2} \\
\Div\bu &=0  &&\qquad \text{in }\Omega\times (0,T], \label{Euler3} \\
\mathbf{u}&=\mathbf{0} &&\qquad\text{on }\partial\Omega\times (0,T], \label{Euler4}
\end{alignat}
where $\be_3:=(0,0,1)$, $\mathbf{u}=(u_1,u_2,u_3)$ is the velocity field, 
$p$ is the pressure, $\theta$ either denotes the temperature (in the case of 
atmosphere) or the density (in the case of ocean) of the fluid in question. 
$\theta_0$ is a reference value of $\theta$.
Also
\[
\frac{D}{Dt}:= \frac{\p}{\p t} +\bu\cdot \nab  
\]
denotes the material derivative. Recall that $\bu^\bot:=(u_2,-u_1,0)$.
Finally, $f$, assumed to be a positive constant, is known as 
{\em the Coriolis parameter}, and $g$ is the gravitational acceleration constant.
We note that the term $f\bu^\bot$ is the so-called Coriolis force which is
an artifact of the earth's rotation (cf. \cite{Salmon}).  

Ignoring the (low order) material derivative term in \eqref{Euler1} we get 
\begin{align} \label{substitution}
\nab_H p &= f \bu^\bot, \\
\frac{\p p}{\p x_3} &= -\frac{\theta}{\theta_0} g, \label{substitution1}
\end{align}
where
\[
\nab_H:=\Bigl(\frac{\p}{\p x_1}, \frac{\p}{\p x_2},0\Bigr).
\]
Equation \eqref{substitution} is known as {\em the geostrophic balance}, 
which describes the balance between the pressure gradient force 
and the Coriolis force in the horizontal directions. Equation \eqref{substitution1} 
is known as {\em the hydrostatic balance} in the literature, which describes the 
balance between the pressure gradient force and the gravitational force 
in the vertical direction.  Define
\begin{equation}\label{wind}
\bu_g:=f^{-1} (\nab p)^\bot\qquad\mbox{and}\qquad
\bu_{ag}:=\bu-\bu_g,
\end{equation}
which are often called {\em the geostrophic wind} and {\em ageostrophic wind},
respectively.

The geostrophic and hydrostatic balances give very simple relations
between the pressure field and the velocity field. However, the dynamics of
the fluids are missing in the description. To overcome this limitation,
J. B. Hoskins \cite{hoskins75} proposed so-called semigeostrophic 
approximation which is based on replacing the material derivative term 
$\frac{D \bu}{D t}$ by $\frac{D \bu_g}{D t}$ in \eqref{Euler1}. This then 
leads to the following semigeostrophic flow equations (in the primitive 
variables):
\begin{alignat}{2}\label{semigeoapprox}
\frac{D\bu_g}{Dt} + (\nab p)^\bot &= f \bu^{\bot} 
&&\qquad \text{in }\Omega\times (0,T],\\
\frac{\p p}{\p x_3} &= -\frac{\theta}{\theta_0} g  
&&\qquad \text{in }\Omega\times (0,T], \label{semigeoapprox1} \\
\frac{D\theta}{Dt} &= 0 &&\qquad \text{in }\Omega\times (0,T], 
\label{semigeoapprox2} \\
\Div\bu &=0 &&\qquad \text{in }\Omega\times (0,T], \label{semigeoapprox3} \\
\bu &= 0  &&\qquad \text{on }\p\Omega\times (0,T].  \label{semigeoapprox4}
\end{alignat}

It is easy to see that after substituting $\bu_g= f^{-1}(\nab p)^\bot$,
\eqref{semigeoapprox} is an evolution equation for $(\nab p)^\bot$.
There are no explicit dynamic equations for $\bu$ in the 
above semigeostrophic flow model. Also, by the definition of the material
derivative, $\frac{D\bu_g}{Dt}=\frac{\p \bu_g}{\p t} 
+ (\bu\cdot \nab) \bu_g$. We note that the full velocity 
$\bu$ appears in the last term. Should $\bu\cdot \nab$ be replaced by 
$\bu_g\cdot \nab$ in the material derivative, the resulting model 
is known as {\em the quasi-geostrophic flow equations} (cf. \cite{Majda}).

Due to the peculiar structure of the semigeostrophic flow equations,
it is difficult to analyze and to numerically solve the equations. 
The first successful analytical approach is the one based on the
fully nonlinear reformulation \eqref{intro1}--\eqref{vdef}, which was 
first proposed in \cite{Brenier91} and was further developed 
in \cite{Benamou_98, Loeper_06} (see \cite{cullen_feldman06} for
a different approach).  The main idea of the reformulation 
is to use time-dependent curved coordinates so the resulting 
system becomes partially decoupled. Apparently, the trade-off
is the presence of stronger nonlinearity in the new formulation.
 
The derivation of the fully nonlinear reformulation 
\eqref{intro1}--\eqref{vdef} starts with introducing 
the so-called {\em geopotential} and {\em geostrophic transformation}
\begin{align}\label{psidefinition}
\psi:=\frac{p}{f^2}+\frac{1}{2}|x_H|^2,\qquad \Phi:=\nab \psi;
\qquad\mbox{where}\qquad x_H:=(x_1,x_2,0).
\end{align}
A direct calculation verifies that
\begin{align*}
\Phi:= x_H + \frac{1}{f^2}(\nab p)^\bot -\frac{\theta}{\theta_0 f^2} g\be_3  
= x_H + \frac{1}{f} \bu_g -\frac{\theta}{\theta_0 f^2} g \be_3,
\end{align*}
consequently,  \eqref{semigeoapprox}--\eqref{semigeoapprox2} can be 
rewritten compactly as 
\begin{align}\label{transform1}
\frac{D \Phi}{Dt} =fJ(\Phi-x),
\end{align}
where
\[
J=\left( \begin{array}{rrr}
           0 & -1 & 0 \\
           1 & 0 & 0 \\
           0 & 0 & 0 
          \end{array} \right).
\]

For any $x\in \Ome$, let $X(x,t)$ denote the fluid particle trajectory originating
from $x$, i.e., 
\begin{align*}
\frac{d X(x,t)}{dt} &=\bu(X(x,t),t) \qquad\forall\, t>0,\\ 
X(x,0) &=x.
\end{align*}
Define the composite function
\begin{align}\label{transform2}
\Psi(\cdot,t):=\Phi(\cdot,t)\circ X(\cdot,t)=\Phi(X(\cdot,t),t) 
=\nab \psi (X(\cdot,t),t).
\end{align}
Then we have from \eqref{transform1}
\begin{align}\label{transform3}
\frac{\p \Psi(x,t)}{\p t} =f J (\Psi(x,t)-X(x,t))=f (\Psi(x,t)-X(x,t))^\bot. 
\end{align}
Since the incompressibility assumption implies $X$ is volume preserving,
\[
\det(\nab X)=1,
\]
which is equivalent to
\begin{equation}\label{weak}
\int_\Ome g(X(x,t)) dx =\int_\Ome g(x) dx \qquad\forall\,g\in C(\overline{\Ome}).
\end{equation}

To summarize, we have reduced  \eqref{semigeoapprox}--\eqref{semigeoapprox3} 
into \eqref{transform2}--\eqref{weak}.  It is easy to see that $\Psi(x,t)$ 
is not unique because one has a freedom in choosing the geopotential
$\psi$. However, Cullen, Norbury, and Purser \cite{cnp91}  
(also see \cite{Cullen_Douglas99, Benamou_98, Loeper_06}) discovered 
the so-called {\em Cullen-Norbury-Purser principle} which says that
$\Psi(x,t)$ must minimize the geostrophic energy at each time $t$. 
A consequence of this minimum energy principle is that the 
geopotential $\psi$ must be a convex function. Using the assumption 
that $\psi$ is convex and Brenier's polar factorization theorem \cite{Brenier91},
Brenier and Benamou \cite{Benamou_98} proved existence 
of such a convex function $\psi$ and a measure preserving mapping 
$X$ which solves \eqref{transform2}--\eqref{weak}.

To relate \eqref{transform2}--\eqref{weak} with \eqref{intro1}, \eqref{intro2},
and \eqref{intro3a}, let $\alpha(y,t)dy$ be the image measure of the Lebesgue 
measure $dx$ by $\Psi(x,t)$, that is
\begin{align*}
\int_\Omega g(\Psi(x,t))dx=\int_{\mathbb{R}^3}
g(y)\alpha (y,t)dy\spa \forall g\in C_c(\mathbb{R}^3).
\end{align*}
We note that the image measure $\alpha(y,t) dy$ is the 
push-forward $\Psi_{\#}dx$ of $dx$ by $\Psi(x,t)$, and $\alpha(y,t)$ is 
the density of $\Psi_{\#}dx$ with respect to the Lebesgue measure $dy$.

Assume that $\psi$ is sufficiently regular, it follows from 
\eqref{transform2} and \eqref{weak} that 
\begin{equation}\label{weak1}
\int_\Ome g(\Psi(x,t)) dx =\int_\Ome g(\nab \psi(X(x,t),t)) dx 
=\int_\Ome g(\nab\psi(x,t)) dx \qquad\forall\,g\in C_c(\bR^3).
\end{equation}
Using a change of variable $y=\nab \psi(x,t)$ on the right and 
the definition of $\alpha(y,t)dy$ on the left we get
\begin{align*}
\int_{\bR^3} g(y) \alpha(y,t) dy 
=\int_{\bR^3} g(y) \det(D^2 \psi^*(y,t)) dy \qquad\forall\,g\in C_c(\bR^3), 
\end{align*}
where $\psi^*$ denotes the Legendre transform of $\psi$, that is,
\begin{align}\label{psistardefinition}
\psi^*(y,t)=\sup_{x\in \Omega}\bigl(x\cdot y-\psi(x,t)\bigr).
\end{align}
Hence
\[
\alpha(y,t) = \det(D^2 \psi^*(y,t)), 
\]
which yields \eqref{intro1}.  

For convex function $\psi$, by a property of the Legendre transform we 
have $\nab\psi^*(y,t)=x\in \Ome$. Hence $\nab\psi^*\subset \Ome$, therefore,
\eqref{intro3a} holds.

Finally, for any $w\in C_c^\infty([-1,T];\mathbb{R}^3)$, it follows
from integration by parts and \eqref{transform3} that
\begin{align*}
&-\int_{\Ome} w(\Psi(x,0),0)\, dx
=\int_0^T\int_{\Ome} \frac{d w(\Psi(x,t),t)}{dt}\, dxdt\\ 
&\qquad\qquad
=\int_0^T\int_{\Ome} 
\Bigl\{ \nabla w(\Psi(x,t),t)\cdot \frac{\p \Psi(x,t)}{\p t} 
+\frac{\p w(\Psi(x,t),t)}{\p t} \Bigr\} \, dxdt\\
&\qquad\qquad
=\int_0^T\int_{\Ome} \Bigl\{ \nabla w(\Psi(x,t),t) 
\cdot f(\Psi(x,t)-X(x,t))^\bot + \frac{\p w(\Psi(x,t),t)}{\p t}\Bigr\}\, dx dt.
\end{align*}
Making a change of variable $y=\nab \psi(x,t)$ and using the definition 
of $\alpha(y,t) dy$ we get
\begin{align}\label{weak2}
\int_0^T\int_{\bR^3} \Bigl\{ \frac{\p w(y,t)}{\p t} 
+f \bv(y,t)\cdot \nabla w(y,t)\Bigr\} \alpha(y,t)\, dy dt
+\int_{\bR^3} w(y,0)\alpha(y,0)\, dy=0,
\end{align}
where $\bv$ is as in \eqref{vdef}. Hence,
\[
\frac{\p \alpha(y,t)}{\p t} +f \Div( \bv(y,t) \alpha(y,t)) = 0, 
\]
which gives \eqref{intro3} as $f=1$ is assumed in Section \ref{sec-1}.

We remark that \eqref{weak1} and \eqref{weak2} are weak formulations of
\eqref{intro1} and \eqref{intro2}, respectively.  We also cite
the following existence and regularity results for \eqref{intro1}-\eqref{intro3}
and refer the reader to \cite{Benamou_98} for their proofs.
\begin{thm}\label{Benamouthm}
Let $\Ome_0,\Omega\subset \bR^3$ be two bounded Lipschitz domain. 
Suppose further that $\alpha_0\in L^p(\bR^3)$ with $\alpha_0\ge 0$, 
$\text{supp}(\alpha_0)\subset\Ome_0$, and 
$\int_{\Ome_0} \alpha_0(x)dx=|\Omega|$. Then for any $T>0$, 
$p>1$,  \eqref{intro1}-\eqref{intro3} has a weak solution 
$(\psi^*,\alpha)$ in the sense of \eqref{weak1} and \eqref{weak2}. 
Furthermore, there exists an $R>0$ such that 
$\mbox{supp}(\alpha(x,t))\subset B_R(0)$ for all $t\in [0,T]$ and 
\begin{alignat*}{2}
&\alpha\in L^\infty([0,T]; L^p(B_R(0))) &&\quad \text{nonnegative},\\
&\psi\in L^\infty([0,T]; W^{1,\infty}(\Omega)) &&\quad \text{convex in physical space},\\
&\psi^*\in L^\infty([0,T]; W^{1,\infty}(\bR^3) &&\quad \text{convex in dual space}.
\end{alignat*}
\end{thm}

\begin{remark}
(a). The above compact support result for $\alpha$ justifies our approach
of solving the original infinite domain problem on a truncated 
computational domain $U$, in particular, if $U$ is chosen large enough
so that $B_R(0)\subset U$.

(b). Since $\alpha$ and $\psi^*$ are not physical variables, one needs 
to recover the physical variables $\bu$ and $p$ from $\alpha$ and $\psi^*$.
This can be done by the following procedure. First, one constructs 
the geopotential $\psi$ from its Legendre transform $\psi^*$. Numerically,
this can be done by fast inverse Legendre transform algorithms. Second, 
one recovers the pressure field $p$ from the geopotential $\psi$
using \eqref{psidefinition}.  Third, one obtains the geostrophic 
wind $\bu_g$ and the full velocity field $\bu$ from the pressure field $p$ 
using \eqref{wind}.

(c). Recently, Loeper \cite{Loeper_06} generalized the above results 
to the case where $\alpha$ is a global weak probability measure solution of the 
semigeostrophic equations.

(d). As a comparison, we recall that two-dimensional incompressible Euler 
equations (in the vorticity-stream function formulation) has the form
\begin{alignat*}{2}
\Delta \phi & = \omega &&\quad\mbox{in } \Ome\times(0,T], \\
\frac{\p \omega}{\p t} + \Div(\bu \omega) 
&= 0 &&\quad\mbox{in } \Ome\times(0,T], \\
\bu &= (\nab \phi)^\bot. &&
\end{alignat*}
Clearly, the main difference is that $\phi$-equation above is a linear 
equation while $\psi^*$ in \eqref{intro2} is a fully nonlinear equation.
\end{remark}

We conclude this section by remarking that in the case that
the gravity is omitted, then the flow becomes two-dimensional. 
Repeating the derivation of this section and dropping the third
component of all vectors, we then obtained a $2$-d semigeostrophic 
flow model which has exactly the same form as \eqref{intro1}--\eqref{vdef}
except that the definition of the operator $(\cdot)^\perp$ becomes
$\bw^\perp:=(w_2,-w_1)$ for $\bw=(w_1,w_2)$, and $\bv$ in \eqref{vdef}
is replaced by 
\[
\bv=(\psi^*_{x_2}-x_2,x_1-\psi^*_{x_1}).
\] 
Similarly, $\bv^\vepsi$ in \eqref{vepsdef} should be replaced by
\[ 
\bv^\eps=(\psi^\eps_{x_2}-x_2,x_1-\psi^\eps_{x_1}).
\]
In the remaining of this paper we shall consider numerical approximations of
both $2$-d and $3$-d models.

\section{Formulation of the numerical method}\label{sec-3}

\subsection{Formulation of the vanishing moment approximation}\label{sec-3.1}
As pointed out in Section \ref{sec-1}, the primary difficulty for 
analyzing and numerically approximating the semigeostrophic 
equations \eqref{intro1}--\eqref{vdef} is caused by the strong 
nonlinearity and non-uniqueness of the $\psi^*$-equation (i.e., Monge-Amp\'ere 
equation. cf. \cite{Aleksandrov61,Gilbarg_Trudinger01}). 
The strong nonlinearity makes the equation non-variational, so any 
Galerkin type numerical methods is not directly applicable to the 
fully nonlinear equation. Non-uniqueness is difficult to deal at the discrete
level because no effective selection criterion is known 
in the literature which guarantees picking up the physical solution
(i.e., the convex solution). Because of the above difficulties,
very little progress was made in the past on developing numerical
methods for the Monge-Amp\'ere equation and other fully nonlinear
second order PDEs (cf. \cite{Dean_Glowinski06b,Oberman07,Oliker_Prussner88}).   

Very recently, we have developed a new approach, called {\em the vanishing 
moment method}, for solving the Monge-Amp\'ere equation and other fully 
nonlinear second order PDEs (cf. \cite{Feng1,Feng2,Feng3,Feng4,Neilan_08,Neilan}).
Our basic idea is to approximate a fully nonlinear second order PDE
by a singularly perturbed quasilinear fourth order PDE.
In the case of the Monge-Amp\'ere equation, we approximate the fully nonlinear
second order equation 
\begin{align}\label{fixedalpha}
\det(D^2 w)=\varphi
\end{align} 
by the following fourth order quasilinear PDE 
\begin{align*}
-\eps\Delta^2 w^\eps +\det(D^2 w^\eps)=\varphi\spa (\eps>0) 
\end{align*}
accompanied by appropriate boundary conditions.
Numerics of \cite{Feng2, Feng3, Feng4,Neilan_08} show that for
fixed $\varphi \ge 0$, $w^\eps$ converges to the unique convex
solution $w$ of \eqref{fixedalpha} as $\eps\to 0^+$. Rigorous proof
of the convergence in some special cases was carried out in \cite{Feng1}.  
Upon establishing the convergence of the vanishing moment method,
one can use various well-established numerical methods (such as
finite element, finite difference, spectral and discontinuous
Galerkin methods) to solve the perturbed quasilinear fourth order PDE. 
Remarkably, our experiences so far suggest that the vanishing moment 
method always converges to the physical solution. The success motivates
us to apply the vanishing moment methodology to the semigeostrophic 
model \eqref{intro1}--\eqref{vdef},  which leads us to studying
problem \eqref{intro4}--\eqref{intro9}.

\begin{remark}
Since a perturbation term is introduced in \eqref{intro4}, 
it is also natural to introduce a ``viscosity" term 
$-\vepsi\Delta \alpha$ on the left-hand side of \eqref{intro5}.  
We believe this should be another viable strategy and will 
further explore the idea and compare the anticipated new result 
with that of this paper. 
\end{remark}

Since \eqref{intro4}--\eqref{intro5} is a quasilinear system,
we can define weak solutions for problem \eqref{intro4}--\eqref{intro9}
in the usual way using integration by parts.

\begin{definition}\label{def3.1}
A pair of functions $(\psi^\vepsi,\alpha^\vepsi)\in L^\infty((0,T);H^2(U))
\times L^2((0,T);H^1(U))\cap H^1((0,T);L^2(U))$ is called a weak solution 
to \eqref{intro4}--\eqref{intro9} if they satisfy 
the following integral identities for almost every $t\in (0,T)$:
\begin{alignat}{2}\label{weakform1}
-\vepsi \bigl( \Del \psi^\vepsi, \Del v\bigr) +\bigl(\det(D^2\psi^\eps), v\bigr)
&=(\alpha^\vepsi, v) + \langle \vepsi^2, v\rangle
&&\qquad\forall v\in H^2(U),\\
\Bigl(\frac{\p \alpha^\vepsi}{\p t}, w\Bigr) 
+\bigl( \bv^\vepsi \cdot \nab \alpha^\vepsi, w\bigr) &=0 &&\qquad\forall w\in H^1(U),
\label{weak_form2}\\
\bigl(\alpha^\eps(\cdot,0),\chi \bigr) &=\bigl(\alpha_0,\chi \bigr) 
&&\qquad\forall \chi\in L^2(U), \label{weak_form3}\\
(\psi^\eps,1)&=0,  &&\qquad  \label{weak_form4}
\end{alignat}
here $\bv^\eps
=(\psi^\eps_{x_2}-x_2,x_1-\psi^\eps_{x_1},0)$ when $d=3$ and 
$\bv^\eps=(\psi^\eps_{x_2}-x_2,x_1-\psi^\eps_{x_1})$ when $d=2$, 
and we have used the fact that $\Div \bv^\vepsi=0$.
\end{definition}

For the continuation of the paper, we assume that there exists a
unique solution to \eqref{intro4}--\eqref{intro9} such that
$\psi^\eps(x,t)$ is convex, $\alpha^\eps(x,t)\geq 0$, and supp
$\alpha^\eps(x,t)\subset B_R(0)\subset U$ for all $t\in [0,T]$. We
also assume $\psi^\eps\in L^2((0,T);H^s(U))$ ($s\ge 3$),
$\alpha^\eps\in L^2((0,T);H^p(U))$ ($p\ge 2$), and that the
following bounds hold (cf. \cite{Feng1}) for almost all $t\in [0,T]$
\begin{align}\label{bounds}
&\|\psi^\eps(t)\|_{H^j}=O(\eps^{\frac{1-j}{2}})\,\, (j=1,2,3),\qquad
\|\Phi^\eps(t)\|_{L^\infty}=O(\eps^{-1}),\\
&\|\psi^\eps(t)\|_{W^{j,\infty}}=O(\eps^{1-j})\,\, (j=1,2),\qquad 
\|\alpha^\eps(t)\|_{W^{1,\infty}}=O(\eps^{-1}), \label{bounds1}
\end{align}
where $\Phi^\eps=\text{cof}(D^2\psi^\eps)$ denotes the cofactor 
matrix of $D^2\psi^\eps$.

As expected, the proof of the above assumptions is extensive
and not easy. We do not intend to give a full proof in this paper. 
However, in the following we shall present a proof for a 
key assertion, that is, $\alpha^\vepsi(x,t)\geq 0$ in $U\times [0,T]$
provided that $\alpha_0(x)\geq 0$ in $\bR^d (\, d=2,3)$. Clearly, this assertion is 
important to ensure that $\psi^\vepsi(\cdot,t)$ is a convex function
for all $t\in [0,T]$.

\begin{proposition}\label{prop1}
Suppose $(\alpha^\vepsi,\psi^\vepsi)$ is a regular solution 
of \eqref{intro4}--\eqref{intro9}. Assume $\alpha_0(x)\geq 0$ in $\bR^d (\, d=2,3)$,
then $\alpha^\vepsi(x,t)\geq 0$ in $U\times [0,T]$. 
\end{proposition}

\begin{proof}
For any fixed $(x,t)\in U\times (0,T]$, let $X^\vepsi(x,t;s)$ denote the 
characteristic curve passing through $(x,t)$ for the transport 
equation \eqref{intro5}, that is.
\begin{align*}
\frac{d X^\vepsi(x,t;s)}{ds} &=\bv^\vepsi(X^\vepsi(x,t;s),s) 
\qquad\quad\forall s\neq t,\\ 
X(x,t;t) &=x.
\end{align*}
Then the solution $\alpha^\vepsi$ at $(x,t)$ can be written as  
\[
\alpha^\vepsi(x,t)=\alpha_0(X^\vepsi(x,t;0)).
\]
Hence, $\alpha^\vepsi(x,t)\geq 0$ for all $(x,t)\in U\times [0,T]$. 
The proof is complete.
\end{proof}

\subsection{Formulation of modified characteristic finite element method}\label{sec-3.2}

Let $\mathcal{T}_h$ be a quasiuniform triangulation or rectangular partition
of $U$ with mesh size $h\in (0,1)$ and $V^h\subset H^2(U)$ denote a 
conforming finite element space (such as Argyris, Bell, Bogner--Fox--Schmit, and
Hsieh--Clough--Tocher finite element spaces \cite{Ciarlet78} when $d=2$)
consisting of piecewise polynomial functions of degree $r \,(\geq 4)$ such that 
for any $v\in H^s(U)\ (s\ge 3)$
\begin{align}
\inf_{v_h\in V^h} \|v-v_h\|_{H^j} \leq h^{\ell-j} \|v\|_{H^s},\quad
j=0,1,2;\, \ell=\min\{r+1,s\}.
\end{align}
Also let $W^h$ be a finite dimensional subspace of $H^1(U))$ consisting
of piecewise polynomials of degree $k\, (\ge 1)$ associated with the mesh $\mathcal{T}_h$.  

Set 
\begin{alignat}{2}\label{spacedef1}
&V^h_0:=\Bigl\{v_h\in V^h;\ \normd{v_h}\Bigl|_{\partial U}=0 \Bigr\},
\qquad && V^h_1:=\{v_h\in V^h_0;\, (v_h,1)=0\},\\
&W^h_0:=\{w_h\in W^h;\ w_h\big|_{\partial U}=0\}, \qquad &&
\tau:=\frac{(1, \bv^\vepsi)}{\sqrt{1+|\bv^\vepsi|^2}} \in \bR^{d+1}. \label{spacedef2}
\end{alignat}
It is easy to check that
\[
\frac{\p }{\p \tau}:=\tau\cdot \Bigl(\frac{\p }{\p t},\nab \Bigr)
= \frac{1}{\sqrt{1+|\bv^\vepsi|^2}}\Bigl(\frac{\p}{\p t} +\bv^\eps\cdot \nab\Bigr).
\]
Hence, from \eqref{intro5} we have
\begin{equation}\label{semgeo2a}
\frac{\p \alpha^\eps}{\p \tau}=\frac{1}{\sqrt{1+|\bv^\vepsi|^2}}
\Bigl(\frac{\p \alpha^\eps}{\p t} +\bv^\eps\cdot \nab \alpha^\eps \Bigr)=0.
\end{equation}
Here we have used the fact that $\Div \bv^\eps=0$.

For a fixed positive integer $M$, let $\Delta t:=\frac{T}{M}$ and
$t_m:=m\Delta t$ for $m=0,1,2,\cdots,M$.  For any $x\in U$, let 
$\bar{x}:=x-\bv^\eps(x,t) \Delta t$. It follows from
the Taylor's formula that (cf. \cite{Douglas_84, Douglas_Russell_82})
\begin{equation}\label{taylor}
\frac{\p \alpha^\vepsi(x, t_m)}{\p \tau}
=\frac{\alpha^\eps(x,t_m)-\alpha^\eps(\bar{x},t_{m-1})}{\Delta t} + O(\Delta t)
\qquad \mbox{for } m=1,2,\cdots,M.
\end{equation}

Borrowing the ideas of \cite{Douglas_84, Douglas_Russell_82}, we 
propose the following modified characteristic finite element
method for problem \eqref{intro4}--\eqref{intro9}:

\medskip
{\bf Algorithm 1:}

{\em Step 1}:  Let $\alpha_h^0$ be the finite element 
interpolation or the elliptic projection of $\alpha_0$.

{\em Step 2}:  For $m=0,1,2,\hdots M$, find $(\psi^m_h,\alpha^{m+1}_h)\in
V_1^h\times W^h_0$ such that
\begin{alignat}{2}\label{geomethod1}
-\eps(\Delta \psi_h^{m},\Delta v_h)+(\det(D^2\psi_h^{m}),v_h)
&=(\alpha_h^{m},v_h)+\langle \eps^2,v_h\rangle  &&\qquad\forall v_h\in V^h_0,\\
(\psi_h^m,1)&=0, &&\qquad \label{geomethod1a} \\ 
\bigl(\alpha_h^{m+1}-\oalpha_h^{m},w_h\bigr) &=0
&&\qquad\forall w_h\in W^h_0, \label{geomethod2}  
\end{alignat} 
where
\[
\oalpha_h^{m}:=\alpha_h^{m}(\bar{x}_h),\qquad
\bar{x}_h:=x-\bv^{m}_h\Del t,\qquad
\bv^m_h:= (\nab \psi^{m}_h-x)^\bot.
\]

In the case that $W^h$ is the continuous linear finite element space (i.e., $k=1$), 
we have the following lemma.

\begin{lemma}\label{lem3.1}
Let $k=1$ in the definition of $W^h$, suppose that $\alpha_h^0 \geq 0$ in $\bR^d 
(\, d=2,3)$, then the solution of Algorithm 1 satisfies $\alpha_h^m \geq 0$ 
in $U$ for all $m\geq 1$.
\end{lemma}

\begin{proof}
In the case $k=1$, \eqref{geomethod2} immediately implies that 
\[
\alpha_h^{m+1}(P_j)=\alpha_h^m(\overline{P}_j),
\]
where $\{P_j\}$ denote the nodal points of the mesh $\cT_h$ and 
$\overline{P}_j:=P_j-\bv^{m}_h\Del t$.  Suppose that $\alpha_h^m(P_j)\geq 0$
for all $j$. Since the basis functions of the linear element are nonnegative,
then we have $\alpha_h^m(\overline{P}_j)\geq 0$ for all $j$. Hence, 
$\alpha_h^{m+1}(P_j)\geq 0$ for all $j$. Therefore, the assertion follows from 
the induction argument.
\end{proof}

\begin{remark} 
The positivity of $\alpha_h^m$ for all $m\geq 1$ gives 
 hope to verify the convexity of $\psi_h^m$, which remains
as an open problem (cf. \cite{Feng2, Feng3, Feng4}).  For high order
finite elements (i.e., $k\geq 2$), $\alpha_h^m$ might take negative values
for some $m>0$ although we shall show later that the deviation from zero
must be very small.
\end{remark}

Let $(\psi^\eps,\alpha^\eps)$ be the solution of
\eqref{intro4}--\eqref{intro9} and $(\psi^m_h,\alpha^m_h)$ be
the solution of \eqref{geomethod1}-\eqref{geomethod2}.  In the
subsequent sections we prove existence and uniqueness for
$(\psi^m_h,\alpha^m_h)$ and provide optimal order error estimates for
$\psi^\eps(t_m)-\psi^m_h$ and $\alpha^\eps(t_m)-\alpha^m_h$ under
certain mesh and time stepping constraints. To this end, we first
study \eqref{geomethod1} independently, which motivates us
to analyze finite element approximations of the Monge-Amp\'ere 
equation with small perturbations of the data. Such an analysis
enables us to bound the error $\psi^\eps(t_m)-\psi^m_h$ in terms of 
of the error $\alpha^\eps(t_m)-\alpha^{m}_h$.  We use similar techniques 
to those developed in \cite{Feng4} to carry out the analysis.  
With this result in hand, we use an inductive argument in 
Section \ref{sec-5} to get the desired error estimates 
for both $\psi^\eps(t_m)-\psi^m_h$ and $\alpha^\eps(t_m)-\alpha^{m}_h$.

\section{Finite element approximations of the Monge-Amp\'ere equation 
with small perturbations}\label{sec-4}

As mentioned above, analyzing the error $\psi^\eps(t_m)-\psi^m_h$
motivates us to consider finite element approximations 
of the following auxiliary problem: for $\vepsi>0$, 
\begin{alignat}{2}\label{uvarphi1}
-\eps\Delta^2 u^\varphi+\det(D^2u^\varphi)&=\varphi\ (>0)\spa &&\text{in }U,\\
\normd{u^\varphi}&=0\spa &&\text{on }\partial U, \label{uvarphi2}\\
\normd{\Delta u^\varphi}&=\eps\spa &&\text{on }\partial U, \label{uvarphi3}\\
(u^\varphi,1)&=0, &&  \label{uvarphi4}
\end{alignat}
whose weak formulation is defined as seeking $u^\varphi\in H^2(\Ome)$ such that
\begin{align}\label{weakMA1}
-\eps \bigl(\Del u^\varphi, \Del v\bigr) + \bigl( \det(D^2u^\varphi), v\bigr)
&=\bigl( \varphi, v\bigr) +\langle \eps^2, v\rangle \qquad\forall v\in H^2(U),\\
(u^\varphi,1)&=0. &&  \label{weakMA2}
\end{align} 

We note that the finite element approximation of a similar 
Monge-Amp\'ere problem was constructed and
analyzed in \cite{Feng4}, where the Dirichlet boundary condition was
considered and the right-hand side function $\varphi$ is the 
same in the finite element scheme as in the PDE problem. 
In this section, we shall study the finite element approximation
of \eqref{uvarphi1}--\eqref{uvarphi4} in which $\varphi$ is 
replaced by $\tilde{\varphi}:=\varphi+\delta \varphi$, where 
$\delta \varphi$ is some small perturbation of $\varphi$.
Specifically, we analyze the following finite element approximation
of \eqref{uvarphi1}--\eqref{uvarphi4}: find $u^\varphi_h\in V^h_1$ 
such that
\begin{align}\label{uvarphih}
-\eps(\Del u^\varphi_h,\Del v_h)+(\det(D^2u^\varphi_h),v_h)
&=(\tilde{\varphi},v_h)+\left\langle
\eps^2,v_h\right\rangle  \spa \forall v_h\in
V^h_0.\end{align}

As expected, we shall adapt the same ideas and techniques as those 
of \cite{Feng4} to analyze the above scheme. However, we shall 
omit some details if they are same as those of \cite{Feng4} but
highlight the differences if they are significant, in particular,
we shall trace how the error constants depend on $\vepsi$ and 
$\delta\varphi$.  Also, since the analysis in $2$-d and $3$-d are essentially
the same, we shall only present the detailed analysis of the 
three dimensional case and make comments about the two 
dimensional case when there is a meaningful difference. 

To analyze scheme \eqref{uvarphih}, we first recall that (cf. \cite{Feng4})
the associated bilinear form of the linearization of the operator 
$M^\eps(u^\varphi):=-\eps\Delta ^2 u^\varphi+\det(D^2u^\varphi)$ at
the solution $u^\varphi$ is given by
\begin{align}\label{Bdef}
B[v,w]:=\eps(\Delta v,\Delta w)+(\Phi^\varphi \nabla v,\nabla w),
\end{align}
where $\Phi^\varphi={\rm cof}(D^2u^\varphi)$ denotes the cofactor matrix of $D^2u^\varphi$.

Next, we define a linear operator $T^\varphi: V^h_1\to V^h_1$ such
that for $w_h\in V^h_1$, $T^\varphi(w_h)\in V^h_1$ is the solution
of following problem:
\begin{align}\label{Toperator}
B[w_h-T^\varphi(w_h),v_h] = \eps(\Delta w_h,\Delta v_h) &-
(\det(D^2 w_h),v_h) \\
&+(\tilde{\varphi},v_h) + \langle\eps^2, v_h\rangle 
\quad \forall v_h\in V^h_{0}. \nonumber
\end{align}

It follows from \cite[Theorem 3.5]{Feng4} that $T^\varphi$ is
well-defined.  Also, it is easy to see that any fixed point of
$T^\varphi$ is a solution to \eqref{uvarphih}.  We now show that if
$\|\delta \varphi\|_\lt$ is sufficiently small, then indeed,
$T^\varphi$ has a unique fixed point in a neighborhood of
$u^\varphi$.  To this end, we set
\[
\mathbb{B}_{h}(\rho):=\bigl\{v_h\in V^h_1;\ \|v_h-I_h u^\varphi\|_\htw
\le \rho \bigr\},
\]
where $I_h u^\varphi$ denotes the finite element interpolant of
$u^\varphi$ onto $V^h_1$.

Before we continue, we state a lemma concerning the divergence row
property of cofactor matrices. A short proof can be found in \cite{evans}.
\begin{lem}\label{cofactor}
Given a vector-valued function $\mathbf{w}=(w_1,w_2,\cdots,w_n):
U\rightarrow \mathbb{R}^n$. Assume $\mathbf{w}\in [C^2(U)]^n$.
Then the cofactor matrix $\text{\rm cof}(D\mathbf{w})$ of the
gradient matrix $D\mathbf{w}$ of $\mathbf{w}$ satisfies the
following row divergence-free property:
\begin{equation}\label{e3.1}
\Div (\text{\rm cof}(D\mathbf{w}))_i =\sum_{j=1}^n \partial_{x_j}
(\text{\rm cof}(D\mathbf{w}))_{ij} =0 \qquad\text{\rm for }
i=1,2,\cdots, n,
\end{equation}
where $(\text{\rm cof}(D\mathbf{w}))_i$ and $(\text{\rm
cof}(D\mathbf{w}))_{ij}$ denote respectively the $i$th row and the
$(i,j)$-entry of $\text{\rm cof}(D\mathbf{w})$.
\end{lem}

Throughout the rest of this section, we assume $u^\varphi\in H^s$,
set $\ell={\rm min}\{r+1,s\}$, and assume the following bounds
(compare to those of \cite{Feng4} and \eqref{bounds}): for $j=1,2,3,$
\begin{align}\label{bounds2}
&\|u^\varphi\|_{H^j}=O(\eps^{\frac{1-j}{2}}),\quad
\|u^\varphi\|_{W^{2,\infty}}=O(\eps^{-1}),\quad
\|\Phi^\varphi\|_{L^\infty}=O(\eps^{-1}).
\end{align}

We then have the following results.

\begin{lem}\label{mapcenterlem}
There exists a constant $C_1(\eps)=O(\eps^{-1})$ such that
\begin{align}\label{mapcenter}\|I_h u^\varphi-T^\varphi(I_hu^\varphi)\|_{H^2}\le
C_1(\eps)\bigl(\eps^{-2}h^{\ell-2}\|u^\varphi\|_{H^\ell}+\|\delta
\varphi\|_{H^{-2}}\bigr).\end{align}
\end{lem}

\begin{proof}
To ease notation set $s_h=I_hu^\varphi-T^\varphi(I_hu^\varphi)$ and
$\eta=I_hu^\varphi-u^\varphi$.  Then for any $v_h\in V^h_0$, we use
the Mean Value Theorem to get
\begin{align*}
B[s_h,v_h]&=\eps(\Delta(I_hu^\varphi,\Delta v_h)-(\det(D^2(I_hu^\varphi),v_h)\\
&\hspace{0.4in}+(\tilde{\varphi},v_h)+\langle \eps^2,v_h\rangle\\
&=\eps(\Delta \eta,\Delta v_h)+(\det(D^2u^\varphi)-\det(D^2(I_hu^\varphi),v_h)
+(\delta \varphi,v_h)\\
&=\eps(\Delta \eta,\Delta v_h)+(\Upsilon^\eps:D^2(u^\varphi-I_hu^\varphi),v_h)
+(\delta \varphi,v_h),
\end{align*} 
where $\Upsilon^\eps={\rm cof}(\tau D^2(I_h u^\varphi)+(1-\tau)D^2u^\varphi)$
for $\tau\in [0,1]$.

On noting that
\[
|\Upsilon_{ij}|=|\text{cof}(\tau D^2(I_h u^\varphi)+(1-\tau)D^2u^\varphi)_{ij}|
=\bigl|\det\bigl(\tau D^2(I_h u^\varphi)\big|_{ij}+(1-\tau)D^2u^\varphi\big|_{ij}\bigr)\bigr|,
\]
where $D^2 u^\varphi\big|_{ij}$ denotes the resulting $2\times 2$ matrix
after deleting the $i^{th}$ row and $j^{th}$ column of $D^2u^\varphi$, we obtain 
\begin{align*}
|(\Psi^\eps)_{ij}|&\le 2 \max_{k\neq i,\ell\neq j}
\left(|\tau (D^2(I_h u^\varphi)_{k\ell}+(1-\tau)(D^2 u^\varphi)_{k\ell}|\right)^2\\
&\le C \max_{k\neq i, \ell\neq j}|(D^2 u^\varphi)_{k\ell}|^2
\le C \|D^2 u^\varphi\|_{L^\infty}^2.
\end{align*}
Hence, from \eqref{bounds2} it follows that 
$\|\Upsilon^\eps\|_{L^\infty}=O(\eps^{-2})$.  Thus,
\begin{align*}
B[s_h,v_h] &\le \eps\|\Delta \eta\|_\lt
\|\Delta v_h\|_\lt + C\eps^{-2}\|D^2 \eta\|_\lt\|v_h\|_\lt
+\|\delta \varphi\|_{H^{-2}}\|v_h\|_\htw\\
&\le C\bigl(\eps^{-2}\|\eta\|_{H^2}+\|\delta
\varphi\|_{H^{-2}}\bigr)\|v_h\|_{H^2}.
\end{align*}

Finally, using the coercivity of $B[\cdot,\cdot]$ we get
\begin{align*}
\|s_h\|_\htw &\le C_1(\eps)\bigl(\eps^{-2}h^{\ell-2}\|u^\varphi\|_{H^\ell}
+\|\delta \varphi\|_{H^{-2}}\bigr).
\end{align*}
The proof is complete.
\end{proof}

\begin{lem}\label{contractinglem}
There exists $h_0>0$ such that for $h\le h_0$, there exists an
$\rho=\rho(h,\eps)$ such that for any $v_h,w_h\in \mathbb{B}_{h}(\rho)$ 
there holds
\begin{align}\label{contracting}
\|T^\varphi(v_h)-T^\varphi(w_h)\|_\htw\le \frac12 \|v_h-w_h\|_\htw.
\end{align}
\end{lem}

\begin{proof}
From the definitions of $T^\varphi(v_h)$ and $T^\varphi(w_h)$ we get
for any $z_h\in V^h_0$
\begin{align*}
B[T^\varphi(v_h)&-T^\varphi(w_h),z_h]\\
&=\left(\Phi^\varphi(\nabla v_h-\nabla w_h),\nabla z_h\right)
+ \left(\text{det}(D^2v_h)-\text{det}(D^2w_h),z_h\right).
\end{align*}

Adding and subtracting $\det(D^2v_h^\mu)$ and $\det(D^2w_h^\mu)$,
where $v_h^\mu$ and $w_h^\mu$ denote the standard mollifications of
$v_h$ and $w_h$, respectively, yields
\begin{align*}
&B[T^\varphi(v_h)-T^\varphi(w_h),z_h] \\
&\quad =(\Phi^\varphi(t_m)(\nabla v_h-\nabla w_h),\nabla z_h)
+(\text{det}(D^2v^\mu_h) -\text{det}(D^2w^\mu_h),z_h)\\
&\hspace{1.75cm}
+(\text{det}(D^2v_h)-\text{det}(D^2v_h^\mu),z_h)+(\text{det}(D^2w^\mu_h)
-\text{det}(D^2w_h),z_h)\\
&\quad =(\Phi^\varphi (\nabla v_h-\nabla w_h),\nabla z_h)
+(\Psi_h:(D^2v^\mu_h-D^2w^\mu_h),z_h)\\
&\hspace{1.75cm}
+(\text{det}(D^2v_h)-\text{det}(D^2v_h^\mu),z_h)
+(\text{det}(D^2w^\mu_h)-\text{det}(D^2w_h),z_h),
\end{align*}
where $\Psi_h=\text{cof}(D^2v_h^\mu+\tau(D^2w_h^\mu-D^2v_h^\mu))$ 
for $\tau\in [0,1]$.

Using Lemma \ref{cofactor} and Sobolev's inequality we have
\begin{align} &\label{lem52x}
B[T^\varphi(v_h)-T^\varphi(w_h),z_h]\\
&\qquad =((\Phi^\varphi-\Psi_h)(\nabla v_h-\nabla w_h),\nabla z_h)
+(\Psi_h(\nabla v_h-\nabla v_h^\mu),\nabla z_h)\nonumber\\
&\qquad\qquad +(\Psi_h(\nabla w_h^\mu-\nabla w_h),z_h)+(\text{det}(D^2v_h)
-\text{det}(D^2v_h^\mu),z_h)\nonumber\\
&\qquad\qquad +(\text{det}(D^2w^\mu_h)-\text{det}(D^2w_h),z_h)\nonumber\\
&\qquad \le C\Bigl\{ \|\Phi^\varphi-\Psi_h\|_\lt\|v_h-w_h\|_\htw
+\|\Psi_h\|_\lt \bigl[\|v_h-v_h^\mu\|_\htw\nonumber\\
&\qquad\qquad +\|w_h-w_h^\mu\|_\htw \bigr]
+\|\text{det}(D^2v_h)-\text{det}(D^2v_h^\mu)\|_\lt\nonumber\\
&\qquad\qquad +\|\text{det}(D^2w_h)-\text{det}(D^2w_h^\mu)\|_\lt \Bigr\}
\|z_h\|_{H^2}.\nonumber
\end{align}
It follows from the Mean Value Theorem that
\begin{align*}
\|(\Phi^\varphi -\Psi_h)_{ij}\|_\lt
&=\|\det(D^2 u^\varphi\big|_{ij})-\det(D^2 v_h^\mu\big|_{ij}
+\tau (D^2 w_h^\mu\big|_{ij}-D^2 v_h^\mu\big|_{ij}))\|_\lt\\
&=\|\Lambda^{ij} :(D^2 u^\varphi \big|_{ij}-(D^2 v_h^\mu\big|_{ij}
+\tau(D^2 w_h^\mu\big|_{ij}-D^2 v_h^\mu\big|_{ij})))\|_\lt,
\end{align*} 
where
$\Lambda^{ij}=\text{cof}(D^2 u^\varphi\big|_{ij}+\lambda(D^2
v_h^\mu\big|_{ij}+\tau(D^2 w_h^\mu\big|_{ij}-D^2
v_h^\mu\big|_{ij})))\in \mathbf{R}^{2\times 2}$ for $\lambda\in [0,1]$.

We bound $\|\Lambda^{ij}\|_{L^\infty}$ as follows:
\begin{align*}
\|\Lambda^{ij}\|_{L^\infty}&=\|\text{cof}(D^2 u^\varphi\big|_{ij}
+\lambda(D^2 v_h^\mu\big|_{ij}+\tau(D^2 w_h^\mu\big|_{ij}-D^2 v_h^\mu\big|_{ij})))
\|_{L^\infty}\\
&=\|D^2 u^\varphi \big|_{ij}+\lambda(D^2 v_h^\mu\big|_{ij}
+\tau(D^2 w_h^\mu\big|_{ij}-D^2 v_h^\mu\big|_{ij}))\|_{L^\infty}\\
&\le C\bigl(\eps^{-1}+h^{-\frac32}\rho+\|D^2v_h^\mu-D^2v_h)\|_{L^\infty}
+\|D^2w_h^\mu-D^2w_h\|_{L^\infty}\bigr),
\end{align*}
where we used the triangle inequality followed by the inverse
inequality and \eqref{bounds2}. Combining the above two inequalities we get 
\begin{align*}
\|(\Phi^\varphi-\Psi_h)_{ij}\|_\lt
&\le \|\Lambda^{ij}\|_{L^\infty}\|D^2 u^\varphi\big|_{ij}-(D^2 v_h^\mu\big|_{ij}
+\tau(D^2 w_h^\mu\big|_{ij}-D^2 v_h^\mu\big|_{ij}))\|_\lt\\
&\le C\bigl(\eps^{-1}+h^{-\frac32}\rho+\|D^2v_h^\mu-D^2v_h\|_{L^\infty}
+\|D^2w_h^\mu-D^2w_h\|_{L^\infty}\bigr)\\
&\quad \times\bigl(h^{\ell-2}\|u^\varphi\|_{H^\ell} +\rho
+\|D^2v_h-D^2v_h^\mu\|_\lt+\|D^2w_h^\mu-D^2w_h\|_\lt\bigr).
\end{align*}
Hence,
\begin{align}\label{phibound}
\|\Phi^\varphi-\Psi_h\|_\lt &\le C\bigl(\eps^{-1}
+h^{-\frac32}\rho+\|D^2v_h^\mu-D^2v_h\|_{L^\infty}
+\|D^2w_h^\mu-D^2w_h\|_{L^\infty}\bigr)\\
&\quad \times\bigl(h^{\ell-2}\|u^\varphi\|_{H^\ell}+\rho+\|D^2v_h-D^2v_h^\mu\|_\lt
+\|D^2w_h^\mu-D^2w_h\|_\lt\bigr). \nonumber 
\end{align}

Applying \eqref{phibound} to \eqref{lem52x} and setting $\mu\to 0$ yield
\begin{align*}
B[T^\varphi(v_h)-T^\varphi(w_h),z_h]
\le C\bigl(\eps^{-1}+h^{-\frac32}\rho\bigr)
\bigl(h^{\ell-2}\|u^\varphi\|_{H^\ell}+\rho \bigr) \|v_h-w_h\|_{H^2} \|z_h\|_{H^2}.
\end{align*}
Using the coercivity of $B[\cdot,\cdot]$ we get
\begin{align}\label{Testimate}
\|T^\varphi(v_h) -T^\varphi(w_h)\|_{H^2}
\le C\eps^{-1}\bigl(\eps^{-1}+h^{-\frac32}\rho\bigr)
\bigl(h^{\ell-2}\|u^\varphi\|_{H^\ell}+\rho\bigr)\|v_h-w_h\|_{H^2}.  
\end{align}

Finally, setting 
$h_0=O\left(\frac{\eps^2}{\|u^\varphi\|_{H^\ell}}\right)^{\frac{1}{\ell-2}}$,
$h\le h_0$, and $\rho=O({\rm min}\{\eps^2,\eps h^\frac32\})$,
it then follows from \eqref{Testimate} that
\begin{align*}
\|T^\varphi(v_h)-T^\varphi(w_h)\|_{H^2}
\le \frac12 \|v_h-w_h\|_\htw\spa \forall v_h,w_h\in \mathbb{B}_{h}(\rho).
\end{align*} The proof is complete
\end{proof}

With the help of the above two lemmas, we are ready to state and prove
our main results of this section.


\begin{thm}\label{errorrelationthm}
Suppose $\|\delta \varphi\|_{H^{-2}}=O({\rm min}\{\eps^3,\eps^2h^{\frac32}\})$. 
Then there exists an $h_1>0$ such that for $h\le h_1$, there exists a unique 
solution $u^\varphi_h\in V^h_1$ solving \eqref{uvarphih}.
Furthermore, there holds the following error estimate: 
\begin{align}\label{errorrelation}
\|u^\varphi-u^\varphi_h\|_\htw 
\le C_2(\eps)\bigl(\eps^{-2}h^{\ell-2}\|u^\varphi\|_{H^\ell}+\|\delta
\varphi\|_{H^{-2}}\bigr)
\end{align} 
with $C_2(\eps)=O(\eps^{-1})$.
\end{thm}

\begin{proof}
To show the first claim, we set
\[
h_1=O\left({\rm min}\left\{\left(\frac{\eps^4}{\|u^\varphi\|_{H^\ell}}
\right)^{\frac{2}{2\ell-7}},\left(\frac{\eps^5}{\|u^\varphi\|_{H^\ell}}
\right)^{\frac{1}{\ell-2}}\right\}\right). 
\] 
Fix $h\le h_1$ and set $\rho_1=2C_1(\eps)\bigl(\eps^{-2}h^{\ell-2}
\|u^\varphi\|_{H^\ell}+\|\delta \varphi\|_{H^{-2}}\bigr)$.
Then we have $\rho_1\le C{\rm min}\{\eps^2,\eps h^\frac32\}$.

Next, let $v_h\in \mathbb{B}_{h}(\rho_1)$.  Using the triangle
inequality and Lemmas \ref{mapcenterlem} and \ref{contractinglem} we get
\begin{align*}
\|I_h u^\varphi - T^\varphi(v_h)\|_{H^2}
&\le \|I_h u^\varphi-T^\varphi(I_hu^\varphi)\|_{H^2}
+ \|T^\varphi(I_h u^\varphi )-T^\varphi(v_h)\|_{H^2}\\
&\le C_1(\eps)\bigl(\eps^{-2}h^{\ell-2}\|u^\varphi\|_{H^\ell}
+\|\delta \varphi\|_{H^{-2}}\bigr)+\frac12 \|I_hu^\varphi-v_h\|_{H^2}\\
&\le \frac{\rho_1}{2}+\frac{\rho_1}{2}=\rho_1.
\end{align*}
Hence, $T^\varphi(v_h)\in \mathbb{B}_{h}(\rho_1)$.  In addition, by
\eqref{contracting} we know that $T^\varphi$ is a contracting
mapping in $\mathbb{B}_{h}(\rho_1)$. Thus, the Brouwer's Fixed Point Theorem
\cite{Gilbarg_Trudinger01} guarantees that there exists a unique
fixed point $u^\varphi_h\in \mathbb{B}_{h}(\rho_1)$ which is a solution
to \eqref{uvarphih}.

Finally, using the triangle inequality we get
\begin{align*}
\|u^\varphi-u^\varphi_h\|_{H^2}&\le \|u^\varphi-I_hu^\varphi\|_{H^2}
+\|I_hu^\varphi-u^\varphi_h\|_{H^2}
\le Ch^{\ell-2}\|u^\varphi\|_{H^\ell}+\rho_1\\
&\le C_2(\eps)\bigl(\eps^{-2}h^{\ell-2}\|u^\varphi\|_{H^\ell}
+\|\delta \varphi\|_{H^{-2}}\bigr).
\end{align*}
\end{proof}


\begin{thm}\label{errorrelation2thm}
In addition to the hypothesis of Theorem \ref{errorrelationthm},
assume that the linearization of $M^\vepsi$ at $u^\varphi$ 
(see \eqref{Bdef}) is $H^3$-regular with the regularity
constant $C_s(\eps)$. Furthermore, assume that 
$\|\delta \varphi\|_{H^{-2}}=O(C^{-1}_s(\eps)\eps^{2})$.
Then there exists an $h_2>0$ such that for $h\le h_2$, there holds
\begin{align}\label{errorrelation2}
\|u^\varphi-u^\varphi_h\|_\ho
&\le C_s(\eps)\Bigl\{C_3(\eps)h^{\ell-1}\|u^\varphi\|_{H^\ell}
+(C_4(\eps)h+1)\|\delta \varphi\|_{H^{-2}}\Bigr\},
\end{align} 
where $C_3(\eps)=C_2(\eps)\eps^{-\frac52}$ and 
$C_4(\eps)=C_2(\eps)\eps^{-\frac12}.$
\end{thm}

\begin{proof} Let $e^\varphi:=u^\varphi-u_h^\varphi$ 
and $u^{\varphi,\mu}_h$ denote a standard mollification 
of $u_h^\varphi$. We note that $e^\varphi$
satisfies the following error equation:
\begin{equation}\label{femerroreqn} 
\eps(\Delta e^\varphi,\Delta z_h)+(\text{det}(D^2u_h^\varphi)
-\text{det}(D^2u^\varphi),z_h)+(\delta \varphi,z_h)=0\spa \forall z_h\in V^h_0.
\end{equation}

Using \eqref{femerroreqn}, the Mean Value Theorem, and Lemma \ref{cofactor} we have
\begin{align}\label{e44}
0=\eps(\Delta e^\varphi,\Delta z_h) &-(\tilde{\Phi}\nabla (u^{\varphi,\mu}_h
-u^\varphi),\nabla z_h)+(\delta \varphi,z_h) \\ 
&+(\text{det}(D^2u^\varphi_h)-\text{det}(D^2u^{\varphi,\mu}_h),z_h),
\nonumber
\end{align}
where
$\tilde{\Phi}=\text{cof}(D^2u^{\varphi,\mu}_h
+\tau(D^2u^\varphi-D^2u^{\varphi,\mu}_h))$ for $\tau\in [0,1]$.

Next, let $\phi\in V_0\cap H^3$ be the unique solution to the following problem:
\begin{align*}
B[\phi,z]=(\nabla e^\varphi,\nabla z)\spa \forall z\in V_0.
\end{align*} 
The regularity assumption implies that
\begin{equation}\label{h3regular}
\|\phi\|_{H^3}\le C_s(\eps)\|\nabla e^\varphi\|_\lt.  
\end{equation}
We then have
\begin{align}\nonumber 
\|\nabla e^\varphi\|_\lt^2 &=\eps(\Delta e^\varphi,\Delta \phi)
+(\Phi^\varphi \nabla \phi,\nabla  e_h^\varphi)\\
\nonumber &=\eps(\Delta e^\varphi,\Delta
(\phi-I_h\phi))+(\Phi^\varphi \nabla e^\varphi_h,\nabla (\phi-I_h\phi))
+\eps(\Delta e^\varphi,\Delta (I_h\phi))\\
\nonumber&\quad +(\Phi^\varphi \nabla e^\varphi,\nabla (I_h\phi))
-\eps(\Delta e^\varphi,\Delta (I_h\phi))
-(\tilde{\Phi}\nabla (u^\varphi-u^{\varphi,\mu}_h),\nabla(I_h\phi))\\
\nonumber&\quad -(\text{det}(D^2 u^\varphi_h) 
-\text{det}(D^2 u^{\varphi,\mu}_h),I_h\phi)-(\delta \varphi,I_h\phi)\\
\label{e44a} &=\eps(\Delta e_h^\varphi,\Delta (\phi-I_h\phi))
+(\Phi^\varphi \nabla e^\varphi,\nabla (\phi-I_h\phi))\\
&\nonumber\quad+((\Phi^\varphi-\tilde{\Phi})\nabla e^\varphi,\nabla(I_h\phi)) 
+(\tilde{\Phi}\nabla (u^{\varphi,\mu}_h-u^\varphi_h),\nabla (I_h\phi))\\
&\nonumber\quad+(\text{det}(D^2u^{\varphi,\mu}_h)
-\text{det}(D^2u^\varphi_h),I_h\phi)-(\delta \varphi,I_h\phi)\\
\nonumber&\le \eps\|\Delta e^\varphi\|_\lt \|\Delta (\phi-I_h\phi)\|_\lt
+ C\|\Phi^\varphi\|_\lt \|e^\varphi\|_\htw\|\phi-I_h\phi\|_\htw\\
\nonumber&\quad +C\|\Phi^\varphi-\tilde{\Phi}\|_\lt\|\nabla e^\varphi\|_\lt
\|\nabla (I_h\phi)\|_{L^\infty}
+C\|\tilde{\Phi}\|_\lt\|u^{\varphi,\mu}_h-u^\varphi_h\|_\htw\|I_h\phi\|_\htw\\
\nonumber&\quad
+\|\text{det}(D^2 u^\varphi_h)-\text{det}(D^2 u^{\varphi,\mu}_h)\|_\lt
\|I_h\phi\|_\lt+\|\delta \varphi_h\|_{H^{-2}}\|I_h\phi\|_\htw\\
&\nonumber\le C\Bigl\{\eps^{-\frac12}h\|e^\varphi\|_\htw
+\|\Phi^\varphi-\tilde{\Phi}\|_\lt\|\nabla e^\varphi\|_\lt
+\|\tilde{\Phi}\|_\lt\|u^{\varphi,\mu}_h-u^\varphi_h\|_\lt\\
\nonumber&\quad +\|\text{det}(D^2 u^\varphi_h)-\text{det}(D^2
u^{\varphi,\mu}_h)\|_\lt+\|\delta \varphi_h\|_{H^{-2}} \Bigr\} \|\phi\|_{H^3}.
\end{align}
We bound as $\|\Phi^\varphi-\tilde{\Phi}\|_\lt$ as follows:
\begin{align*}
\|(\Phi^\varphi-\tilde{\Phi})_{ij}\|_\lt
&=\|\text{cof}(D^2u^\varphi)_{ij}-\text{cof}((D^2u^{\varphi,\mu}_h
+\tau(D^2u^\varphi-D^2u^{\varphi,\mu}_h))_{ij})\|_\lt\\
&=\|\det(D^2u^\varphi\big|_{ij})-\det(D^2u^{\varphi,\mu}_h\big|_{ij}
+\tau(D^2u^\varphi\big|_{ij}-D^2u^{\varphi,\mu}_h\big|_{ij}))\|_\lt\\
&=\|\Lambda^{ij}:(D^2u^\varphi\big|_{ij}-(D^2u^{\varphi,\mu}_h\big|_{ij}
+\tau(D^2u^\varphi\big|_{ij}-D^2u^{\varphi,\mu}_h\big|_{ij})))\|_\lt\\
&\le 2\|\Lambda^{ij}\|_{L^\infty}\bigl(\|D^2u^\varphi-D^2u^\varphi_h\|_\lt
+\|D^2u^\varphi_h-D^2u^{\varphi,\mu}_h\|_\lt\bigr),
\end{align*}
where
$\Lambda^{ij}=\text{cof}(D^2u^\varphi\big|_{ij}
+\lambda(D^2u_h^{\varphi,\mu}\big|_{ij}
+\tau(D^2u^\varphi\big|_{ij}-D^2u^{\varphi,\mu}_h\big|_{ij}))),\ \lambda\in [0,1]$.
Notice that we have abused the notation $\Lambda^{ij}$ by defining it
differently in two proofs.

To estimate $\|\Lambda^{ij}\|_{L^\infty}$, we note that $\Lambda^{ij}\in
\mathbf{R}^{2\times 2}$.  Thus for $h\le h_1$
\begin{align*}
\|\Lambda^{ij}\|_{L^\infty}& \le C\big( \|D^2u^\varphi\|_{L^\infty} + \|D^2u^\varphi-D^2u^\varphi_h\|_{L^\infty}+\|D^2u^\varphi_h-D^2u^{\varphi,\mu}_h\|_{L^\infty}\big)\\ 
&\le C\Bigl(\eps^{-1}+C_2(\eps)\bigl(\eps^{-2}h^{\ell-\frac72}\|u^\varphi\|_{H^\ell}
+h^{-\frac32}\|\delta \varphi_h\|_{H^{-2}}\bigr)
+\|D^2u_h^{\varphi,\mu}-D^2u^\varphi_h\|_{L^\infty}\Bigr)\\
&\le C\bigl(\eps^{-1}+\|D^2u_h^{\varphi,\mu}-D^2u^\varphi_h\|_{L^\infty}\bigr).
\end{align*}
where we have used the triangle inequality, the inverse inequality, 
and \eqref{bounds2}. Therefore,
\begin{align}\label{e4.8a}
\|\Phi^\varphi-\tilde{\Phi}^\eps\|_\lt 
&\le C\Bigl(\eps^{-1}+\|D^2u_h^{\varphi,\mu}-D^2u^\varphi_h\|_{L^\infty}\Bigr)\\
&\nonumber \quad\times
\Bigl(C_2(\eps)\bigl(\eps^{-2}\|u^\varphi\|_{H^\ell}+\|\delta
\varphi_h\|_{H^{-2}}\bigr)+\|D^2u^\varphi_h-D^2u^{\varphi,\mu}_h\|_\lt\Bigr).
\end{align}
Using \eqref{e4.8a} and setting $\mu\to 0$ in \eqref{e44a} yield
\begin{align*}
\|\nabla e^\varphi\|_\lt^2&\le C\Bigl\{\eps^{-\frac12}h\|e^\varphi\|_{H^2}
+\|\delta \varphi\|_{H^{-2}}\\ 
&\hspace{0.4in}+\eps^{-1}C_2(\eps)\bigl(\eps^{-2}h^{\ell-2}\|u^\varphi\|_{H^\ell}
+\|\delta\varphi_h\|_{H^{-2}}\bigr)\|\nabla e^\varphi\|_\lt\Bigr\}\|\phi\|_{H^3}.
\end{align*}
It follows from \eqref{h3regular} that
\begin{align*}
\|\nabla e^\varphi\|_\lt&\le C_s(\eps)\Bigl\{\eps^{-\frac12}h\|e_h^\varphi\|_{H^2}
+\|\delta\varphi\|_\lt\\
&\hspace{0.4in}+\eps^{-1}C_2(\eps)\bigl(\eps^{-2}h^{\ell-2}\|u^\varphi\|_{H^\ell}
+\|\delta \varphi_h\|_{H^{-2}}\bigr)\|\nabla e^\varphi\|_\lt\Bigr\}. 
\end{align*}
Set
$h_2=O\left(\frac{\eps^4}{C_s(\eps)
\|u^\varphi\|_{H^\ell}}\right)^{\frac{1}{\ell-2}}$.
On noting that $\|\delta \varphi\|_{H^{-2}}\le \eps(C_s(\eps)C_2(\eps))^{-1}$ 
we have for $h\le {\rm min}\{h_1,h_2\}$
\begin{align*}
\|\nabla e^\varphi\|_\lt
&\le C_s(\eps)\Bigl(C_2(\eps)\eps^{-\frac52}h^{\ell-1}\|u^\varphi\|_{H^\ell}
+\bigl(\eps^{-\frac12}C_2(\eps)h+1\bigr)\|\delta \varphi_h\|_{H^{-2}}\Bigr).
\end{align*}
Thus, \eqref{errorrelation2} follows from Poincare's inequality.
The proof is complete.
\end{proof}

\begin{remark}\label{uboundremark}
Let $(\psi^m_h,\alpha^m_h)$ be generated by Algorithm 1. 
If $\|\alpha^\eps(t_m)-\alpha^{m}_h\|_\lt=O\bigl({\rm min}
\{\eps^3,\eps^2h^{\frac32},C^{-1}_s(\eps)\eps^2\}\bigr)$, then
by Theorems \ref{errorrelationthm} and \ref{errorrelation2thm}, for
$h\le {\rm min}\{h_1,h_2\}$, there exists a unique $\psi^m_h\in
V^h_1$ solving \eqref{geomethod1}, where
\begin{align*}
h_1&=O\left({\rm min}\left\{\left(\frac{\eps^4}{\|\psi^\eps\|_{L^2([0,T];H^\ell)}}
\right)^{\frac{2}{2\ell-7}},\left(\frac{\eps^5}{\|\psi^\eps\|_{L^2([0,T];H^\ell)}}
\right)^{\frac{1}{\ell-2}}\right\}\right),\\
h_2&=\left(\frac{\eps^4}{C_s(\eps)\|\psi^\eps\|_{L^2([0,T];H^\ell)}}
\right)^{\frac{1}{\ell-2}}.
\end{align*}
Furthermore, we have the following error bounds:
\begin{align}
\|\psi^\eps(t_m)-\psi^m_h\|_{H^2}
&\le C_2(\eps)\bigl(\eps^{-2}h^{\ell-2}\|\psi^\eps(t_m)\|_{H^\ell}
+\|\alpha^\eps(t_m)-\alpha^m_h\|_\lt\bigr),\\
\label{errorrelation3}\|\psi^\eps(t_m)-\psi^m_h\|_{H^1}
&\le C_s(\eps)\bigl(C_3(\eps)h^{\ell-1}\|\psi^\eps(t_m)\|_{H^\ell}
\\
&\hskip 1.0in
+(C_4(\eps)h+1)\|\alpha^\eps(t_m)-\alpha^{m}_h\|_\lt \bigr).
\nonumber
\end{align} 
\end{remark}

\begin{remark}
In the two dimensional case, 
\begin{align*}h_1=O\left(\frac{\eps^{\frac52}}{\|\psi^\eps\|_{L^2([0,T];H^\ell)}}\right)^\frac{1}{\ell-2},\qquad
h_2=O\left(\frac{\eps^3}{C_s(\eps)\|\psi^\eps\|_{L^2([0,T];H^\ell)}}\right)^{\frac{1}{\ell-2}},\end{align*}
\begin{align*}\|\psi^\eps(t_m)-\psi^m_h\|_{H^2}
&\le C_2(\eps)\bigl(\eps^{-\frac12}h^{\ell-2}\|\psi^\eps(t_m)\|_{H^\ell}
+\|\alpha^\eps(t_m)-\alpha^m_h\|_\lt\bigr),\end{align*}
and \eqref{errorrelation3} holds with $C_3(\eps)=C_2(\eps)\eps^{-\frac12}$ and $C_4(\eps)=C_2(\eps)\eps^{-\frac12}.$
Furthermore, we only require $\|\alpha^\eps(t_m)-\alpha^m_h\|_\lt=O({\rm min}\{\eps^2,C_s^{-1}(\eps)\eps\})$.

\end{remark}

\section{Error analysis for Algorithm 1}\label{sec-5}
In this section we shall derive error estimates for the solution of
Algorithm 1. This will be done by using an inductive argument
based on the error estimates of the previous section.
Before stating our first main result of this section, we cite the 
well-known error estimate results for the elliptic projection of $\alpha(t_m)$,
which we denote by $\chi_h^m\in W_0^h$. Let
$\omega^m(\cdot):=\alpha(\cdot,t_m)-\chi_h^m(\cdot)$, then there hold the following 
estimates for $\alpha(t_m)-\chi_h^m, m \ge 1$ (cf. \cite{Brenner, Ciarlet78}):
\begin{align}\label{projbounds1}
\|\omega\|_{L^2([0,T];L^2)}+ h\|\omega\|_{L^2([0,T];H^1)}
&\le Ch^{j}\|\alpha\|_{L^2([0,T];H^j)},\\
\nonumber\|\omega_t\|_{L^2([0,T];L^2)}+h\|\omega_t\|_{L^2([0,T];H^1)}&\le
Ch^j\|\alpha_t\|_{L^2([0,T];H^j)},\\
\nonumber \|\omega^m\|_{W^{1,\infty}}&\le C
h^{j-1}\|\alpha(t_m)\|_{W^{j,\infty}},
\end{align}
where $j:={\rm min}\{k+1,p\}$.  As in Section \ref{sec-4}, we set
$\ell={\rm min}\{r+1,s\}$.  

\begin{thm}\label{maintheorem} 
There exists $h_3>0$ such that for $h\le {\rm min}\{h_1,h_2,h_3\}$ 
there exists $\Delta t_1>0$ such that for
$\Delta t\le {\rm min}\{\Delta t_1,h^2\}$ 
\begin{align}\label{L2boundthm}
\max_{0\le m\le M}\|\alpha^\eps(t_{m})-\alpha^{m}_h\|_\lt
&\le C_5(\eps)\Bigl\{\Delta t \|\alpha_{\tau\tau}^\eps\|_{L^2([0,T] \times \bR^3)}\\
&\nonumber\hspace{0.6in}+h^j\bigl[\|\alpha^\eps\|_{L^2([0,T];H^j)}
+\|\alpha^\eps_t\|_{L^2([0,T];H^j)}\bigr]\\
&\nonumber\hspace{0.6in}+C_6(\eps)h^{\ell}\|\psi^\eps\|_{L^2([0,T];H^\ell)}\Bigr\},\\
\label{H2bound2thm}\max_{0\le m\le M}\|\psi^\eps(t_{m})-\psi^{m}_h\|_{H^2}
&\le C_7(\eps)\Bigl\{\Delta t \|\alpha_{\tau\tau}^\eps\|_{L^2([0,T]\times \bR^3)}\\
&\nonumber\hspace{0.6in}+h^j\bigl[\|\alpha^\eps\|_{L^2([0,T];H^j)}
+\|\alpha^\eps_t\|_{L^2([0,T];H^j)}\bigr]\\
&\nonumber\hspace{0.6in}+C_6(\eps)h^{\ell-2}\|\psi^\eps\|_{L^2([0,T];H^\ell)}\Bigr\},\\
\label{H1bound2thm}\max_{0\le m\le M}\|\psi^\eps(t_{m})-\psi^{m}_h\|_{H^1}
&\le C_8(\eps)\Bigl\{\Del t \|\alpha_{\tau\tau}^\eps\|_{L^2([0,T] \times \bR^3)}\\
&\nonumber\hspace{0.6in}+h^{j}\bigl[\|\alpha^\eps\|_{L^2([0,T];H^j)}
+\|\alpha^\eps_t\|_{L^2([0,T];H^j)}\bigr]\\
&\nonumber\hspace{0.6in}+C_6(\eps)h^{\ell-1}\|\psi^\eps\|_{L^2([0,T];H^\ell)}\Bigr\},
\end{align} 
where $C_5(\eps)=O(\eps^{-1})$, $C_6(\eps)=C_s(\eps)C_3(\eps)$,
$C_7(\eps)=C_2(\eps)C_5(\eps)$, $C_8(\eps)=C_s(\eps)C_5(\eps)$, and $C_s(\eps)$ is defined in Theorem \ref{errorrelation2thm}.
\end{thm}

\begin{proof}
We break the proof into five steps.

{\em Step 1:}  The proof is based on two induction hypotheses, where
we assume for $m=0,1,\cdots,k,$
\begin{align}\label{induction1}
\|\alpha^\eps(t_m)-\alpha^m_h\|_\lt
&=O\bigl({\rm min}\{\eps^3,\eps^2h^{\frac32},C^{-1}_s(\eps)\eps^2\}\bigr),\\
\label{induction2}\|D^2\psi^m_h\|_{L^\infty}&=O(\eps^{-1}).
\end{align}

We first show that the claims of the theorem hold when $k=0$. Let
\[
h_4=O\left({\rm min}\left\{\left(\frac{\eps^3}{\|\alpha_0\|_{H^j}}
\right)^{\frac{1}{j}},\left(\frac{\eps^2}{\|\alpha_0\|_{H^j}}
\right)^{\frac{2}{2j-3}},\left(\frac{\eps^2}{C_s(\eps)\|\alpha_0\|_{H^j}}
\right)^{\frac{1}{j}}\right\}\right).
\]
From \eqref{projbounds1} we have for $h\le h_4$
\begin{align*}
\|\alpha_0-\alpha^0_h\|_\lt \le Ch^{j}\|\alpha_0\|_{H^j}
\le C{\rm min}\{\eps^3,\eps^2h^{\frac32},C^{-1}_s(\eps)\eps^2\}.
\end{align*}

By Remark \ref{uboundremark}, there exists $\psi^0_h$ solving
\eqref{geomethod1}.  On noting that
$h_1\le C \left(\frac{\eps^2}{\|\psi^\eps(0)\|_{H^\ell}}\right)^{\frac{2}{2\ell-7}}$,
we have for $h\le {\rm min}\{h_1,h_2,h_4\}$
\begin{align*}\|D^2\psi^0_h\|_{L^\infty}&\le
\|D^2\psi^\eps(0)\|_{L^\infty}+h^{-\frac32}\|D^2\psi^\eps(0)-D^2\psi^0_h\|_\lt\\
&\le C\Bigl(\eps^{-1}+h^{-\frac32}C_2(\eps)\bigl(\eps^{-2}h^{\ell-2}
\|\psi^\eps(0)\|_{H^\ell}+h^j\|\alpha^\eps_0\|_{H^j}\bigr)\Bigr)
\le C\eps^{-1}.
\end{align*}

The remaining four steps are devoted to show that the estimates hold for $m=k+1$.  

\medskip
{\em Step 2:}
Let $\xi^m:=\alpha_h^m-\chi_h^m$. By \eqref{geomethod2} and \eqref{intro5}, 
and a direct calculation we get
\begin{align}\label{start}
\bigl(\xi^{m+1}-\oxi^{m},\xi^{m+1}\bigr)&= \bigl(\Delta t\,
\alpha_\tau^\eps(t_{m+1})-(\alpha^\eps(t_{m+1})
-\oalpha^\eps_h(t_{m})),\xi^{m+1}\bigr)\\
&\nonumber\hspace{0.4in}+\bigl(\omega^{m+1}-\oomega^{m}_h,\xi^{m+1}\bigr),
\end{align}
where $\oxi^{m}:= \xi^{m}(\bar{x}_h)$, 
$\oalpha^\eps_h(t_{m})):= \alpha^\eps(\bar{x}_h,t_{m})$, and
$\oomega^{m}_h:=\omega^{m}(\bar{x}_h)$.

We now estimate the right-hand side of \eqref{start}. To bound the first term, 
we write
\begin{align*}
&\Del t\, \alpha^\eps_t(x,t_{m+1})-\bigl[\alpha^\eps(x,t_{m+1})
-\alpha^\eps(\bar{x}_h,t_{m})\bigr]\\ 
&\hspace{0.3in}=\Delta t\,
\alpha^\eps_t(x,t_{m+1})-\bigl[\alpha^\eps(x,t_{m+1})
-\alpha^\eps(\bar{x},t_{m})\bigr]+\bigl[\alpha^\eps(\bar{x}_h,t_{m})
-\alpha^\eps(\bar{x},t_{m})\bigr].
\end{align*}
Using the identity
\begin{align*}
\Del t\,\alpha^\eps_\tau(x,t_{m+1}) &-\bigl[\alpha^\eps(x,t_{m+1})
-\alpha^\eps(\bar{x},t_{m})\bigr] \\
&=\int_{(\bar{x},t_{m})}^{(x,t_{m+1})} \sqrt{|x(\tau)-\bar{x}|^2
+(t(\tau)-t_{m})^2}\ \alpha^\eps_{\tau\tau}d\tau
\end{align*}
and \eqref{bounds} we obtain
\begin{align}\label{bound1a}
\|\Del t\, \alpha^\eps_\tau(t_{m+1}) &-\bigl[\alpha^\eps(t_{m+1})
-\oalpha^\eps(t_{m})\bigr]\|_\lt^2\\
&=\int_{\bR^3}\Bigl|\int_{(\bar{x},t_{m})}^{(x,t_{m+1})}
\sqrt{|x(\tau)-\bar{x}|^2+(t(\tau)-t_{m})^2}\ \alpha^\eps_{\tau\tau}d\tau\Bigr|^2 dx
\nonumber \\
&\le \Delta t\int_{\bR^3} \sqrt{|\bv^\eps(t_{m+1})|^2+1} \Bigl|
\int_{(\bar{x},t_{m})}^{(x,t_{m+1})} \alpha^\eps_{\tau\tau}d\tau \Bigr|^2 dx 
\nonumber\\
&\le C\Del t^2\|\bv^\eps(t_{m+1})\|_{L^\infty}
\int_{\mathbb{R}^3}\int_{(\bar{x},t_{m})}^{(x,t_{m+1})}
\bigl|\alpha^\eps_{\tau\tau}\bigr|^2d\tau dx \nonumber\\
&\le C\Delta t^2 \|\alpha^\eps_{\tau\tau}\|^2_{L^2([t_{m},t_{m+1}]
\times \bR^3)}, \nonumber
\end{align}
where $\oalpha^\eps(t_{m}):=\alpha^\eps(\bar{x},t_{m})$.  Since
\begin{align*}
\alpha^\eps(\bar{x}_h,t_{m})-\alpha^\eps(\bar{x},t_{m})
=\int_0^1 D\alpha^\eps(\bar{x}_h+s(\bar{x}-\bar{x}_h),t_m)
\cdot (\bar{x}-\bar{x}_h)ds,
\end{align*}
then
\begin{align}\label{bound1b}
\|\oalpha^\eps_h(t_{m}) &-\oalpha^\eps(t_{m})\|_\lt^2\\
&\nonumber =\Delta t^2 \int_{\bR^3}
\Bigl|\int_0^1 D\alpha^\eps(\bar{x}_h+s(\bar{x}-\bar{x}_h),t_{m})\cdot
(\bv_h^{m}-\bv^\eps(t_{m})) ds\Bigr|^2 dx\\
&\nonumber  \le \Delta
t^2\|\alpha^\eps(t_{m})\|_{W^{1,\infty}}^2\|\bv_h^{m} -\bv^\eps(t_{m})\|_\lt^2\\
&\nonumber \le C\eps^{-2}\Del t^2\|\bv_h^m-\bv^\eps(t_{m})\|_\lt^2.
\end{align}

Using \eqref{bound1a}-\eqref{bound1b}, we can bound the first
term on the right-hand side of \eqref{start} as follows:
\begin{align}\label{firstbound}
&\bigl(\Delta t\, \alpha_\tau^\eps(t_{m+1})-\bigl[\alpha^\eps(t_{m+1})
-\oalpha^\eps_h(t_{m})\bigr],\xi^{m+1}\bigr)\\
&\hspace{0.5in}\nonumber 
\le C\Delta t^2 \bigl(\|\alpha_{\tau\tau}\|_{L^2([t_{m},t_{m+1}]
\times\bR^3)}^2+\eps^{-2}\|\bv_h^{m}-\bv^\eps(t_{m})\|_\lt^2\bigr)
+\frac18\|\xi^{m+1}\|_\lt^2.
\end{align}

To bound the second term on the right-hand side of \eqref{start}, writing
\begin{align*}
\omega^{m+1}(x) &-\omega^{m}(\bar{x}_h)\\
&=\bigl(\omega^{m+1}(x)-\omega^{m}(x)\bigr)
+\bigl(\omega^{m}(x)-\omega^{m}(\bar{x})\bigr)
+\bigl(\omega^{m}(\bar{x})-\omega^{m}(\bar{x}_h)\bigr),
\end{align*}
we then have
\begin{align}\label{bound2a}
\|\omega^{m+1}-\omega^{m}\|_\lt^2 
&\le \Delta t \|\omega_t\|_{L^2([t_{m},t_{m+1}]\times \bR^3)}^2.
\end{align}
Next, it follows from
\begin{align*}
\omega^{m}(x)-\omega^{m}(\bar{x}) 
&=\Delta t \int_0^1 D\omega^{m}(x+s(\bar{x}-x))\cdot \bv^\eps(t_{m})ds
\end{align*}
that (set $\oomega^{m}:=\omega^{m}(\bar{x})$)
\begin{align}\label{bound2b}
\|\omega^{m}-\oomega^{m}\|_\lt^2
&\le C\Del t^2\|\bv^\eps(t_{m})\|_{L^\infty}^2\|\omega^{m}\|_{H^1}^2
\le C\Del t^2\|\omega^{m}\|_{H^1}^2.
\end{align}
Finally, using the identity
\begin{align*}
\omega^{m}(\bar{x})-\omega^{m}(\bar{x}_h)
=\Delta t \int_0^1 D\omega^{m}(\bar{x}+s(\bar{x}_h-\bar{x}))
\cdot(\bv^\eps(t_{m})-\bv^{m}_h) ds
\end{align*}
we get
\begin{align}\label{bound2c}
\|\oomega^{m}-\oomega^{m}_h\|_\lt^2
&\le C\Del t^2 \|\omega^{m}\|_{W^{1,\infty}}^2\|\bv^\eps(t_{m})-\bv^{m}_h\|_\lt^2\\
&\nonumber \le C\Del t^2 \|\bv^\eps(t_{m})-\bv^{m}_h\|_\lt^2.
\end{align}

Combining \eqref{bound2a}--\eqref{bound2c}, we then bound the second
term on the right-hand side of \eqref{start} as follows:
\begin{align}\label{secondbound}
\bigl(\omega^{m+1}-\oomega^{m}_h,\xi^{m+1}\bigr)
&\le C\Bigl(\Delta t \|\omega_t\|_{L^2([t_{m},t_{m+1}]\times
\bR^3)}^2+\Delta t^2\|\omega^{m}\|_{H^1}^2\\
&\nonumber\hspace{0.4in}
+\Delta t^2 \|\bv^\eps(t_{m})-\bv^{m}_h\|_\lt^2\Bigr)
+\frac18 \|\xi^{m+1}\|_\lt^2.
\end{align}

{\em Step 3:} To get a lower bound of
$(\xi^{m+1}-\oxi^{m},\xi^{m+1})$, let 
$F_m(x):=x-\Delta t \bv^m_h(x)$.  We then have
\[
\det(J_{F_m})=1+\Delta t^2\left(1+\psi_{x_1x_1}^m
\psi_{x_2x_2}^m-(\psi_{x_1x_2}^m)^2-(\psi_{x_1x_1}^m+\psi_{x_2x_2}^m)\right),
\]
where $J_{F_m}$ denotes the Jacobian of $F_m$, and we have omitted
the subscript $h$ for notational convenience.  Letting $\Del t_0=O(\eps)$, 
we can conclude from the induction hypotheses that for $\Del t\le \Del t_0$, 
$F_m$ is invertible and $\det(J_{F_m^{-1}})=1+C\eps^{-2}\Delta t^2$. 
From this result we get
\begin{align}\label{transform}
\|\oxi^{m}\|_\lt^2=(1+\eps^{-2}\Delta t^2)\|\xi^{m}\|_\lt^2.
\end{align}
Thus,
\begin{align}\label{lowerbound}
(\xi^{m+1}-\oxi^{m},\xi^{m+1})
&\ge \frac12 \bigl[(\xi^{m+1},\xi^{m+1})-(\oxi^{m},\oxi^{m})\bigr]
\\
\nonumber &=\frac12\bigl(\|\xi^{m+1}\|_\lt^2
-(1+\eps^{-2}\Delta t^2)\|\xi^{m}\|_\lt^2\bigr).
\end{align}

{\em Step 4:} Combining \eqref{start}, \eqref{firstbound}, 
\eqref{secondbound}, \eqref{lowerbound}, and using the induction 
hypotheses and Remark \ref{uboundremark} yield
\begin{align*}
&\|\xi^{m+1}\|_\lt^2-\|\xi^{m}\|_\lt^2\\
&\nonumber\hspace{0.6in}\le C\eps^{-2}\Bigl\{\Delta
t^2\Bigl(\|\alpha^\eps_{\tau\tau}\|_{L^2([t_{m},t_{m+1}]\times\bR^3)}^2
+ \|\omega^{m}\|_{H^1}^2 +\|\bv_h^{m}-\bv^\eps(t_{m})\|_\lt^2\Bigr)\\
&\nonumber\hspace{.8in}
+\Delta t \|\omega_t\|_{L^2([t_{m},t_{m+1}]\times \bR^3)}^2
+\Delta t^2\|\xi^{m}\|_\lt^2\Bigr\}\\
&\nonumber\hspace{0.6in} 
\le C\eps^{-2}\Bigl\{\Delta t^2\Bigl(\|\alpha^\eps_{\tau\tau}\|_{L^2([t_{m},t_{m+1}]
\times\bR^3)}^2+ \|\omega^{m}\|_{H^1}^2\\
&\nonumber\hspace{.8in}
+C^2_s(\eps)\bigl(C^2_3(\eps)h^{2\ell-2}\|\psi^\eps(t_m)\|^2_{H^\ell}
+(C^2_4(\eps)h^2+1)\|\alpha^\eps(t_m)-\alpha_h^m\|^2_{H^{-2}}\bigr)\Bigr)\\
&\nonumber\hspace{.8in}
+\Delta t \|\omega_t\|_{L^2([t_{m},t_{m+1}]\times \bR^3)}^2
+\Delta t^2\|\xi^{m}\|_\lt^2\Bigr\}.
\end{align*}

It follows from the inequality 
\[
\|\alpha^\eps(t_m)-\alpha^m_h\|_{H^{-2}}\le \|\alpha^\eps(t_m)-\alpha^m_h\|_\lt
\le \|\xi^m\|_\lt+\|\omega^m\|_\lt
\]
that
\begin{align*}
\|\xi^{m+1}\|_\lt^2-\|\xi^{m}\|_\lt^2
&\le C\eps^{-2} \Bigl\{\Delta t^2 \Bigl(\|\alpha^\eps_{\tau\tau}
\|_{L^2([t_{m},t_{m+1}]\times\bR^3)}^2+(C_4^2(\eps)h^2+1)\|\omega^{m}\|_{H^1}^2\\
&\nonumber\quad
+C^2_s(\eps)C_3^2(\eps)h^{2\ell-2}\|\psi^\eps(t_m)\|_{H^\ell}^2\Bigr)
+\Del t \|\omega_t\|^2_{L^2([t_m,t_{m+1}]\times \bR^3)}\\
&\nonumber\quad
+(C^2_4(\eps)h^2+1)\Del t^2\|\xi^m\|_\lt^2\Bigr\}.
\end{align*}

Applying the summation operator $\sum_{m=0}^k$ and noting that $\xi^0=0$ we get
\begin{align*}
\|\xi^{k+1}\|_\lt^2&\le C\eps^{-2} \Bigl\{\Delta t^2
\|\alpha^\eps_{\tau\tau}\|_{L^2([0,T]\times\bR^3)}^2
+\Del t\Bigl[(C^2_4(\eps)h^2+1)\|\omega\|_{L^2([0,T];H^1)}^2\\
&\nonumber\hspace{1in}
+C^2_s(\eps)C_3^2(\eps)h^{2\ell-2}\|\psi^\eps\|_{L^2([0,T];H^\ell)}^2
+\|\omega_t\|^2_{L^2([0,T]\times \bR^3)}\Bigr]\\
&\nonumber\hspace{1in}
+\Del t^2(C^2_4(\eps)h^2+1)\sum_{m=0}^k\|\xi^m\|_\lt^2\Bigr\},
\end{align*}
which by an application of the discrete Gronwall inequality yield
\begin{align}\label{Gronwall}
\|\xi^{k+1}\|_\lt &\le C\eps^{-1}\left(1+\eps^{-1}(C_4(\eps)h+1)
\Del t\right)^{k+1}\Bigl\{\Del t\|\alpha^\eps_{\tau\tau}\|_{L^2([0,T] \times\bR^3)}\\
\nonumber&\hspace{0.25in}
+\sqrt{\Del t}\Bigl[(C_4(\eps)h+1)\|\omega\|_{L^2([0,T];H^1)}\\
&\nonumber\hspace{0.6in}
+C_s(\eps)C_3(\eps)h^{\ell-1}\|\psi^\eps\|_{L^2([0,T];H^\ell)}
+\|\omega_t\|_{L^2([0,T]\times \bR^3)}\Bigr]\Bigr\}.
\end{align}

We note $h_1=O(C_4^{-1}(\eps))=O(\eps^{\frac32})$.  Thus, for 
$h\le {\rm min}\{h_1,h_2,h_4\}$ and $\Delta t\le {\rm min}\{\Delta
t_0,h^2\}$, we have from \eqref{Gronwall}, the triangle inequality,
and \eqref{projbounds1} that
\begin{align}\label{L2bound}
\|\alpha^\eps(t_{k+1}) -\alpha^{k+1}_h\|_\lt 
& \le C_5(\eps)\Bigl\{\Delta t \|\alpha_{\tau\tau}^\eps\|_{L^2([0,T]
\times \bR^3)}+h^j\bigl[\|\alpha^\eps\|_{L^2([0,T];H^j)}
\\
&\qquad +\|\alpha^\eps_t\|_{L^2([0,T];H^j)}\bigr]
+C_6(\eps)h^{\ell}\|\psi^\eps\|_{L^2([0,T];H^\ell)}\Bigr\}. \nonumber
\end{align}

Thus, by Remark \ref{uboundremark}, we obtain the following estimates:
\begin{align}\label{H2bound2}
\|\psi^\eps(t_{k+1})-\psi^{k+1}_h\|_{H^2}
&\le C_7(\eps)\Bigl\{\Del t \|\alpha_{\tau\tau}^\eps\|_{L^2([0,T]
\times \bR^3)}+h^j\bigl[\|\alpha^\eps\|_{L^2([0,T];H^j)} \\
&\quad +\|\alpha^\eps_t\|_{L^2([0,T];H^j)}\bigr] 
+C_6(\eps)h^{\ell-2}\|\psi^\eps\|_{L^2([0,T];H^\ell)}\Bigr\}, \nonumber\\
\label{H1bound2}\|\psi^\eps(t_{k+1})-\psi^{k+1}_h\|_{H^1}
&\le C_8(\eps)\Bigl\{\Delta t \|\alpha_{\tau\tau}^\eps\|_{L^2([0,T]
\times \bR^3)}+h^{j}\bigl[\|\alpha^\eps\|_{L^2([0,T];H^j)} \\
&\quad +\|\alpha^\eps_t\|_{L^2([0,T];H^j)}\bigr]
+C_6(\eps)h^{\ell-1}\|\psi^\eps\|_{L^2([0,T];H^\ell)}\Bigr\}. \nonumber
\end{align}

{\em Step 5:}  We now verify the induction hypotheses.  Set
\begin{align*}
C_9(\eps)&=\eps^2{\rm min}\{\eps,C^{-1}_s(\eps)\},\\
C_{10}(\eps)&=C_5(\eps)(\|\alpha^\eps\|_{L^2([0,T];H^j)}
+\|\alpha^\eps_t\|_{L^2([0,T];H^j)}),\\
C_{11}(\eps) &=C_5(\eps)C_6(\eps)\|\psi^\eps\|_{L^2([0,T];H^\ell)},
\end{align*}
and let
\begin{align*}
h_5&=O\Bigl({\rm min}\Bigl\{\left(\frac{C_9(\eps)}{C_{11}(\eps)}
\right)^{\frac{1}{\ell}},\left(\frac{\eps^2}{C_{11}(\eps)}\right)^{\frac{2}{2\ell-3}}
\left(\frac{C_9(\eps)}{C_{10}(\eps)}\right)^{\frac{1}{j}},
\left(\frac{\eps^2}{C_{10}(\eps)}\right)^{\frac{2}{2j-3}}\Bigr\}\Bigr),\\
\Delta t_1&=O\left(\frac{{\rm min}\{C_9(\eps),\eps^2h^{\frac32}\}}{C_5(\eps)
\|\alpha_{\tau\tau}^\eps\|_{L^2([0,T]\times \bR^3)}}\right).
\end{align*}

On noting that $\Delta t_1\le \Delta t_0$, it follows from \eqref{L2bound}
that for $h\le {\rm min}\{h_1,h_2,h_4,h_5\}$ and 
$\Delta t\le {\rm min}\{\Delta t_1,h^2\}$
\begin{align*}
\|\alpha^\eps(t_{k+1})-\alpha^{k+1}_h\|_\lt 
\le C{\rm min}\{\eps^3,\eps^2h^\frac32,C^{-1}_s(\eps)\}.
\end{align*}
Thus, the first induction hypothesis \eqref{induction1} holds.

Finally, let 
\[
h_6=O\left(\frac{\eps}{C_6(\eps)C_7(\eps)\|\psi^\eps\|_{L^2([0,T];H^\ell)}}
\right)^{\frac{2}{2\ell-7}},
\]
by the definitions of $h_5,h_7$ and $\Delta t_1$, \eqref{H2bound2}, 
\eqref{bounds}, and the inverse inequality we have
for $h\le {\rm min}\{h_1,h_2,h_4,h_5,h_6\}$ and
$\Delta t\le {\rm min}\{\Delta t_1,h^2\}$
\begin{align*}
\|D^2\psi^{k+1}_h\|_{L^\infty}
&\le \|D^2\psi^\eps(t_{k+1})\|_{L^\infty}
+Ch^{-\frac32}\|D^2\psi^\eps(t_{k+1})-D^2\psi^{k+1}_h\|_\lt\\
&\le C\eps^{-1}+h^{-\frac32}C_7(\eps)\Bigl\{\Del t 
\|\alpha_{\tau\tau}^\eps\|_{L^2([0,T]\times \bR^3)}\\
&\quad
+h^j\bigl[\|\alpha^\eps\|_{L^2([0,T];H^j)}
+\|\alpha^\eps_t\|_{L^2([0,T];H^j)}\bigr]
+C_6(\eps)h^{\ell-2}\|\psi^\eps\|_{L^2([0,T];H^\ell)}\Bigr\}\\
&\le C\eps^{-1}.
\end{align*}
Therefore, the second induction hypothesis \eqref{induction2} holds,
and the proof is complete by setting $h_3={\rm min}\{h_1,h_2,h_4,h_5,h_6\}$.
\end{proof}

\begin{remark}
In the two dimensional case,
\begin{align*}\Delta t_1=O\left(\frac{{\rm min}\{ \eps^2,C_s^{-1}(\eps)\eps\}}{C_5(\eps)\|\alpha^\eps_{\tau\tau}\|_{L^2([0,T]\times \mathbb{R}^2)}}\right).\end{align*}
\end{remark}

\begin{remark}
Recalling the definitions of $V^h$ and $W^h$, we require $k\ge r-2$ in order 
to obtain optimal order error estimate for $\psi_h^m$ in the $H^2$-norm.
\end{remark}

\section{Numerical experiments}\label{sec-6}
In this section we shall present three $2$-d numerical experiments.
The first two experiments are done on the domain $U=(0,1)^2$, while
the third experiment uses $U=(0,6)^2$. In all three experiments
the fifth degree Argyris plate finite element (cf. \cite{Ciarlet78}) is
used to form $V^h$, and the cubic Lagrange element is employed   
to form $W^h$. We recall that (see Section \ref{sec-2}) the $2$-d geostrophic
flow model has the exact same form as \eqref{intro1}--\eqref{vdef}
except $\bv$ and $\bv^\vepsi$ in \eqref{vdef} and \eqref{vepsdef} 
are replaced respectively by
\[
\bv =(\psi^*_{x_2}-x_2,x_1-\psi^*_{x_1}),
\qquad
\bv^\eps =(\psi^\eps_{x_2}-x_2,x_1-\psi^\eps_{x_1}).
\]

\subsection{Test 1} 
The purpose of this test is twofold. First, we compute 
$\alpha^m_h$ and $\psi^m_h$ to view certain properties of these two 
functions.  Specifically, we want to verify $\alpha^m_h>0$ and
that $\psi^m_h$ is strictly convex for $m=0,1,...,M$. Second,
we calculate $\|\psi^*-\psi^\eps_h\|$ and
$\|\alpha-\alpha^\eps_h\|$ for fixed $h=0.023$ and $\Delta t=0.0005$ in order to
approximate $\|\psi^*-\psi^\eps\|$ and $\|\alpha-\alpha^\eps\|$.
We set to solve problem \eqref{geomethod1}--\eqref{geomethod2} with the 
right-hand side of \eqref{geomethod2} being replaced by $(F,w_h)$, and 
$V^h_1$ and $W^h_0$ being replaced by $V^h_{g_N}$ and $W^h_{g_D}$, 
respectively, where
\begin{align*}
V^h_{g_N}(t) &:=\Bigl\{ v_h\in V^h;\ \normd{v_h}\Big|_{\partial U}=g_N, \ 
(v_h,1)=c(t) \Bigr\},\qquad c(t):=(\psi^*,1), \\
W^h_{g_D}(t) &:=\{w_h\in W^h;\ w_h\big|_{\partial U}=g_D\}.
\end{align*}

We use the following test functions and parameters
\begin{align*}
g_N(x,t) &=te^{t(x_1^2+x_2^2)/2}(x_1\nu_{x_1}+x_2\nu_{x_2}),\\
g_D(x,t) &=t^2(1+t(x_1^2+x_2^2))e^{t(x_1^2+x_2^2)},\\
F(x,t) &=t\bigl(2+4t(x_1^2+x_2^2)+t^2(x_1^2+x_2^2)^2\bigr)e^{t(x_1^2+x_2^2)},
\end{align*}
so that the exact solution of \eqref{intro1}--\eqref{vdef} is given by
\[
\psi^*(x,t)=e^{t(x_1^2+x_2^2)/2},\qquad
\alpha(x,t)=t^2(1+t(x_1^2+x_2^2))e^{t(x_1^2+x_2^2)}.
\]

We record the computed solutions and plot the errors versus $\eps$ 
in Figure \ref{figtest1} at $t_m=0.25$.  The figure shows that 
$\|\psi^*(t_m)-\psi^m_h\|_{H^2}=O(\eps^\frac14)$, and since we have set
both $h$ and $\Delta t$ very small, these results suggest that
$\|\psi^*(t_m)-\psi^\eps(t_m)\|_{H^2}=O(\eps^\frac14)$. Similarly, 
we argue $\|\psi^*(t_m)-\psi^\eps(t_m)\|_{H^1}=O(\eps^\frac34)$ 
and $\|\psi^*(t_m)-\psi^\eps(t_m)\|_\lt=O(\eps)$ based on our results.  
We note that these are the same convergence results found 
in \cite{Feng2,Feng3,Feng4,Neilan_08}, where the single Monge-Amp\'ere 
equation was considered.  We also notice that this test suggests that
$\|\alpha(t_M)-\alpha^\eps_h(t_m)\|_\lt$ may not converge, which suggests
that the convergence can only be possible in a weaker norm such as $H^{-2}$.

Next, we plot $\alpha^m_h$, and $\psi^m_h$ for $t_m=0.1,\ 0.4,$ 
and $1.0$ with $h=0.05$, $\Delta t=0.1$ in Figure
\ref{computedpsi}.  The figure shows that 
$\alpha^\eps_h(t_m)>0$ and the computed solution $\psi^m_h$ 
is clearly convex for all $t_m$.

\begin{figure}[ht]
\begin{center}
\includegraphics[angle=0,width=7.25cm,height=6.25cm]{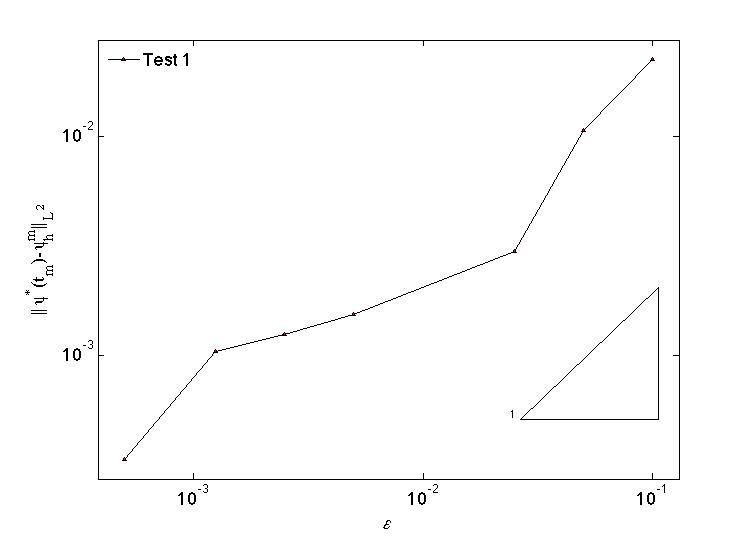}
\includegraphics[angle=0,width=7.25cm,height=6.25cm]{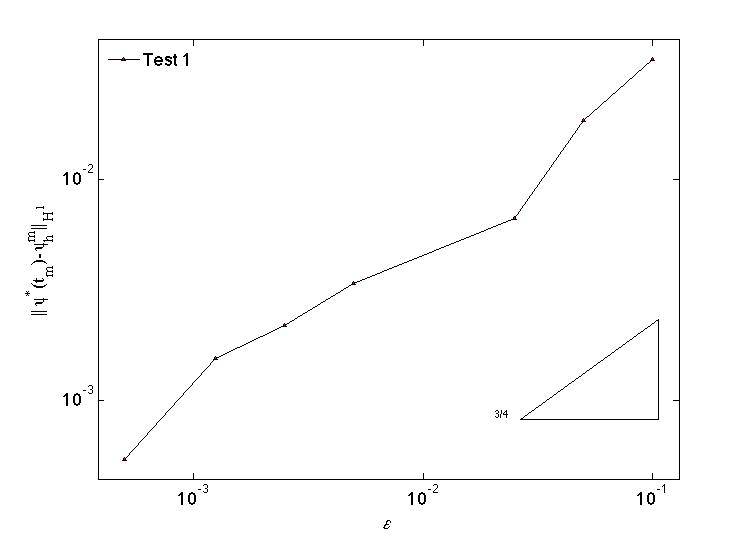}\\
\includegraphics[angle=0,width=7.25cm,height=6.25cm]{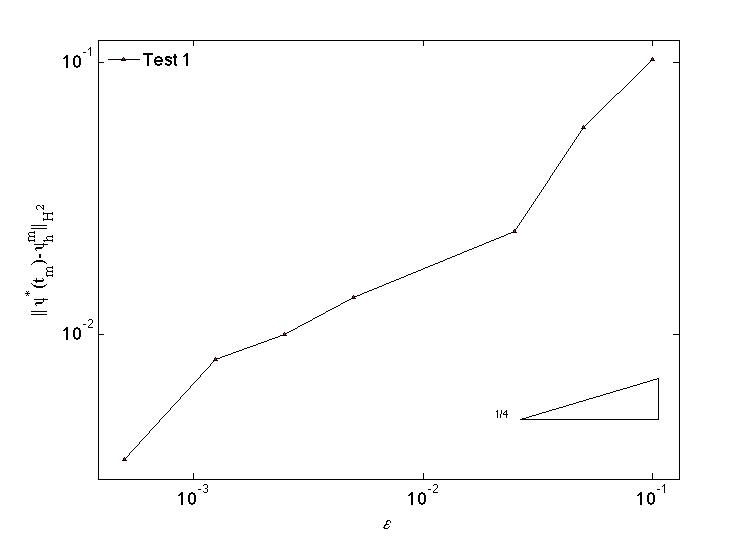}
\includegraphics[angle=0,width=7.25cm,height=6.25cm]{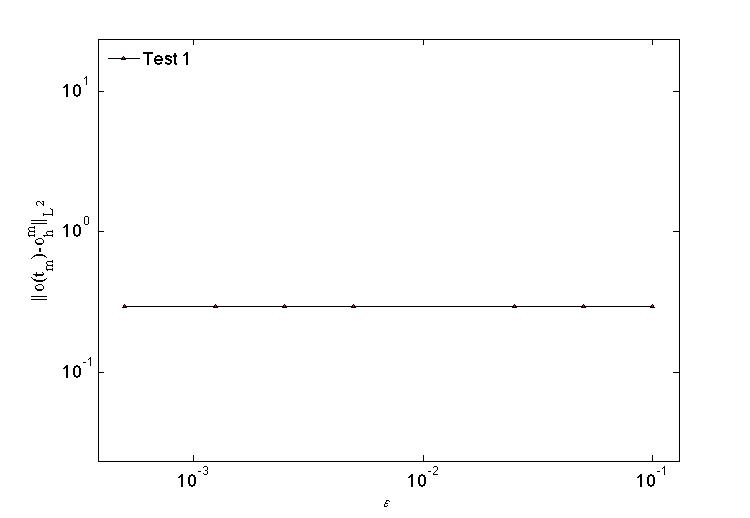}
\end{center}
\caption{\label{figtest1}{\scriptsize Test 1:  Change of
$\|\psi^*(t_m)-\psi^m_h\|$ w.r.t. $\eps$. $h=0.023$, $\Delta
t=0.0005$, $t_m=0.25$.}}
\end{figure}

\begin{figure}[ht]
\begin{center}
\includegraphics[angle=0,width=7.25cm,height=6.25cm]{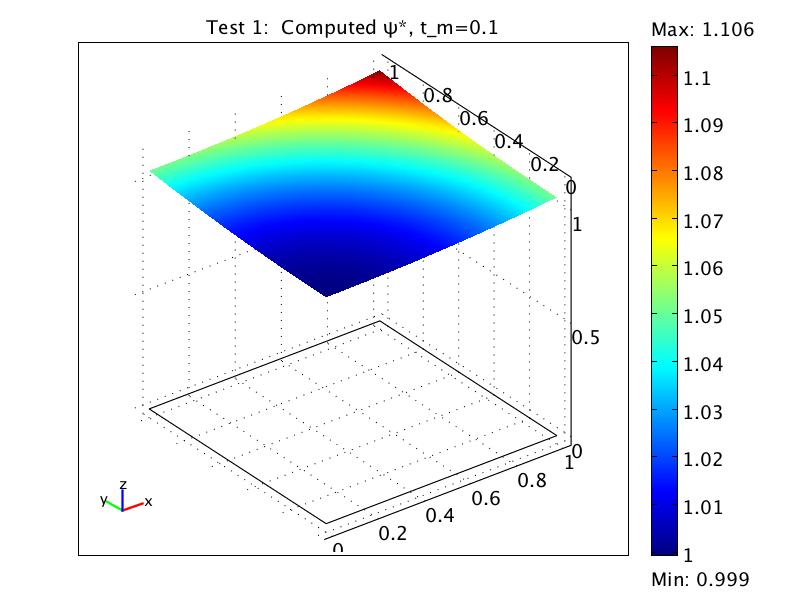}
\includegraphics[angle=0,width=7.25cm,height=6.25cm]{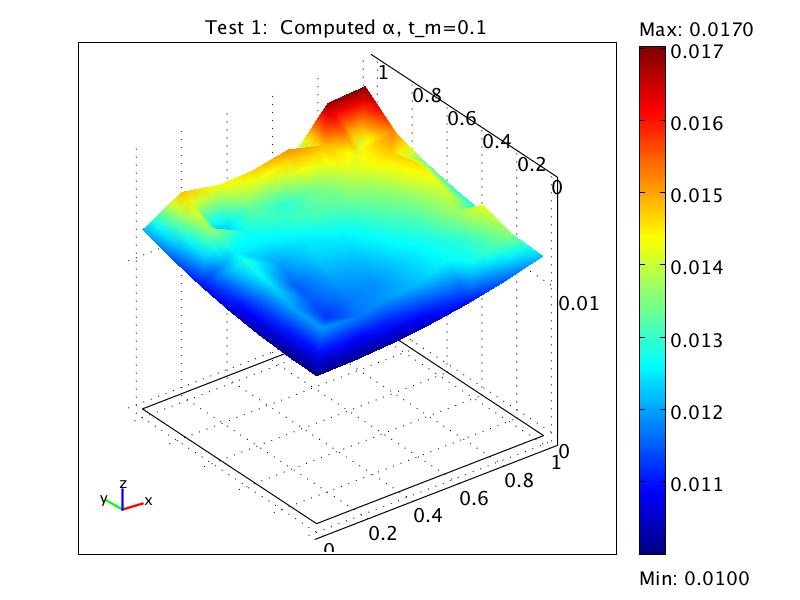}\\
\includegraphics[angle=0,width=7.25cm,height=6.25cm]{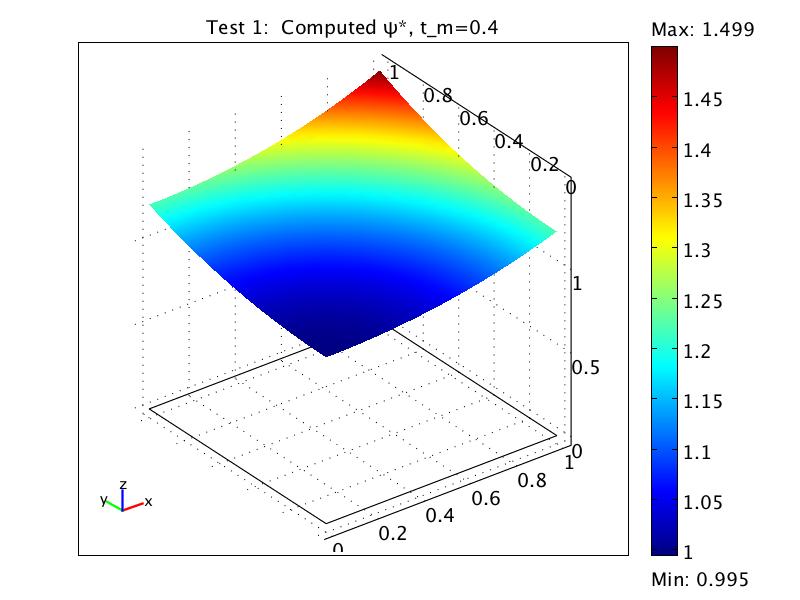}
\includegraphics[angle=0,width=7.25cm,height=6.25cm]{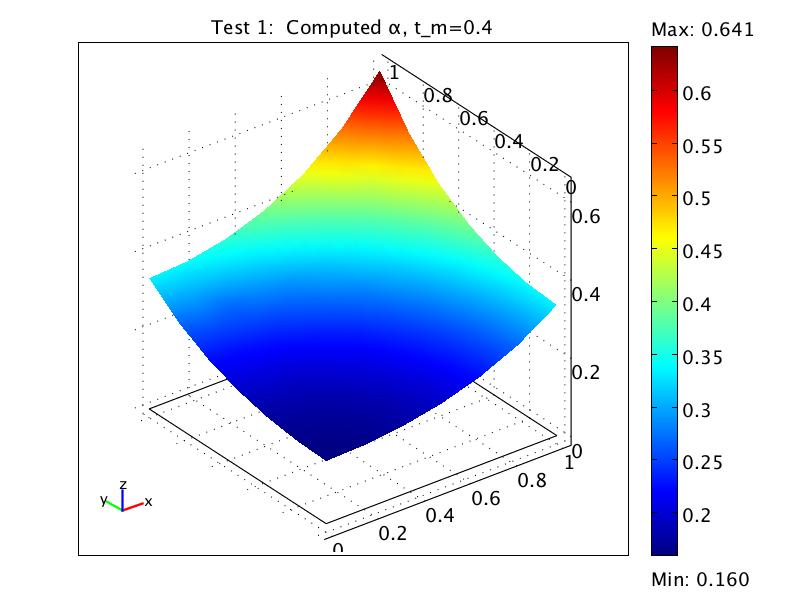}\\
\includegraphics[angle=0,width=7.25cm,height=6.25cm]{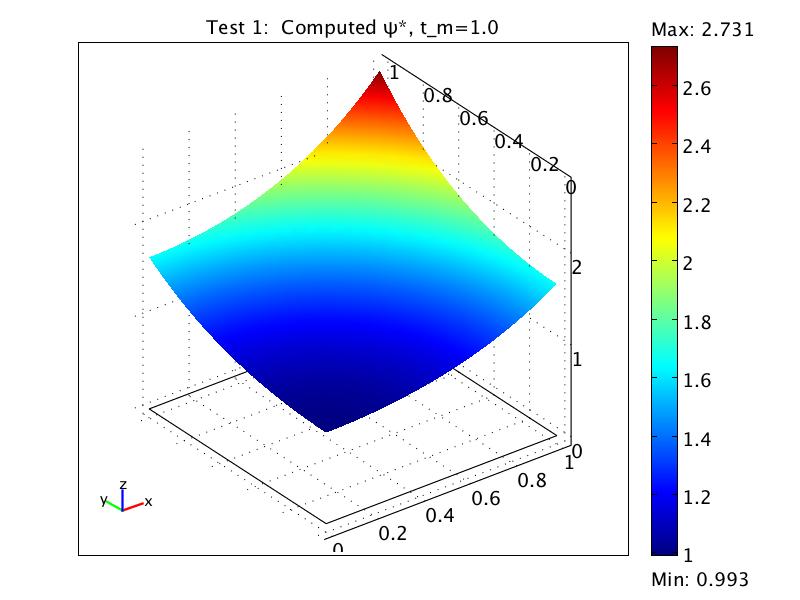}
\includegraphics[angle=0,width=7.25cm,height=6.25cm]{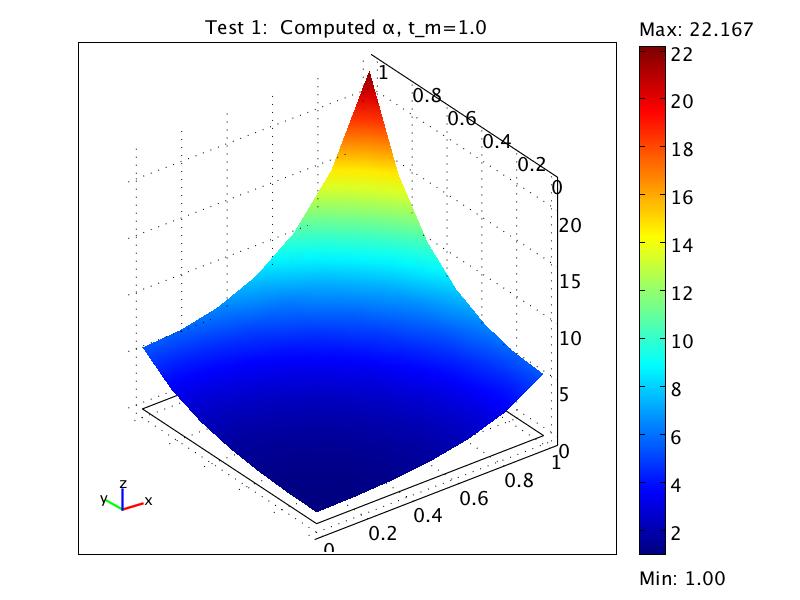}
\end{center}
\caption{\label{computedpsi}{\scriptsize Computed $\psi^m_h$ (left)
and $\alpha^m_h$ (right) for Test 1 
at $t_m=0.1$ (top), $t_m=0.4$ (middle), 
and $t_m=1.0$ (bottom).  $\Delta t=0.1$, $h=0.05$.}}
\end{figure}

\subsection{Test 2} 
The goal of this test is to calculate the rate of 
convergence of $\|\psi^\eps-\psi_h^\eps\|$ and $\|\alpha^\eps-\alpha^\eps_h\|$ 
for a fixed $\eps$ while varying $\Delta t$ and $h$ with the relation 
$\Delta t=h^2$.  We solve \eqref{geomethod1}--\eqref{geomethod2} but 
with a new boundary condition:  
$\frac{\partial \Delta \psi^\eps}{\partial \nu} =\phi^\eps$.  
Let $V^h_{g_N}$ and $W^h_{g_D}$ be defined in the same way as in
Test 1 using the following test functions and parameters
\begin{align*}
c(t) &=(\psi^\vepsi,1),\\
g_N &=te^{t(x_1^2+x_2^2)/2}(x_1\nu_{x_1}+x_2\nu_{x_2}), \\
g_D &=t^2(1+t(x_1^2+x_2^2))e^{t(x_1^2+x_2^2)}, \\
&\qquad -\eps t^2e^{t(x_1^2+x_2^2)/2}(8+8t(x_1^2+x_2^2)+t^2(x_1^2+x_2^2)^2), \\
F &=t\bigl(2+4t(x_1^2+x_2^2)+t^2(x_1^2+x_2^2)^2\bigr)e^{t(x_1^2+x_2^2)} \\
&\qquad -\frac{\eps t}{2}e^{t(x_1^2+x_2^2)/2}\bigl(32+56(x_1^2+x_2^2)t
+16t^2(x_1^2+x_2^2)^2+t^3(x_1^2+x_2^2)^3\bigr), \\
\phi^\eps&=\bigl(\bigl(4x_1t^2+x_2t^3(x_1^2+x_2^2)\bigr)\nu_{x_1}
+\bigl(4x_2t^2+x_2t^3(x_1^2+x_2^2)\nu_{x_2}\bigr)\bigr)e^{t(x_1^2+x_2^2)/2},  
\end{align*}
so that the exact solution of \eqref{intro4}--\eqref{intro9} is given by
\begin{align*}
\psi^\eps(x,t)&=e^{t(x_1^2+x_2^2)/2}, \\
\alpha^\eps(x,t)&=t^2(1+t(x_1^2+x_2^2))e^{t(x_1^2+x_2^2)} 
-\eps t^2e^{t(x_1^2+x_2^2)/2}(8+8t(x_1^2+x_2^2)+t^2(x_1^2+x_2^2)^2).
\end{align*}

The errors at time
$t_m=0.25$ are listed in Table 1 and are plotted verses
$\Delta t$ in Figure \ref{figtest4}.  The results clearly indicate that 
$\|\alpha^\eps(t_m)-\alpha^m_h\|_\lt=O(\Delta t)$ and 
$\|\psi^\eps(t_m)-\psi^m_h\|=O(\Delta t)$ in all norms as expected 
by the analysis in the previous section.  
\begin{figure}[ht]
\begin{center}
\includegraphics[angle=0,width=7.25cm,height=6.25cm]{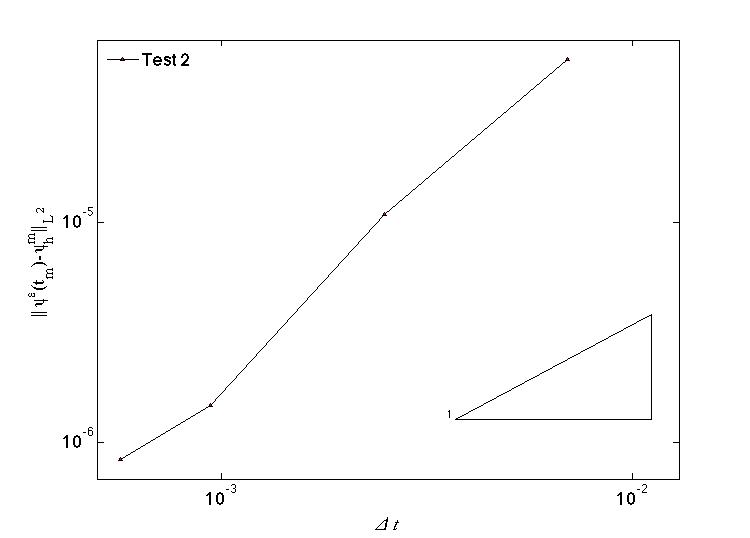}
\includegraphics[angle=0,width=7.25cm,height=6.25cm]{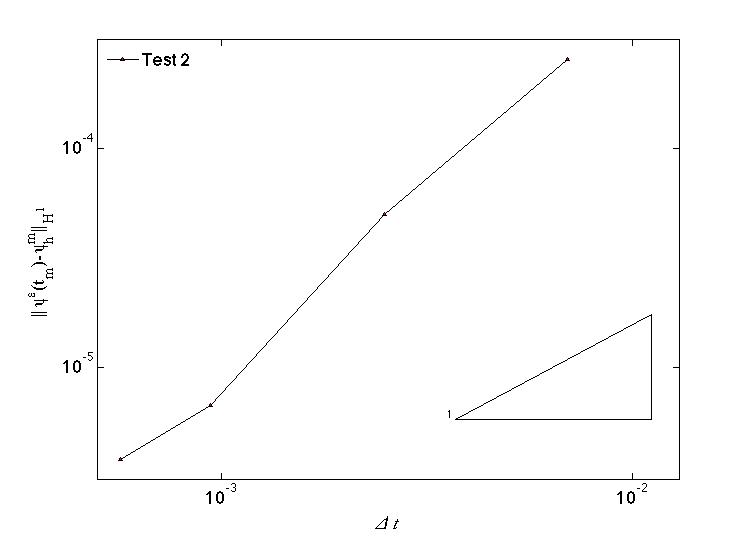}\\
\includegraphics[angle=0,width=7.25cm,height=6.25cm]{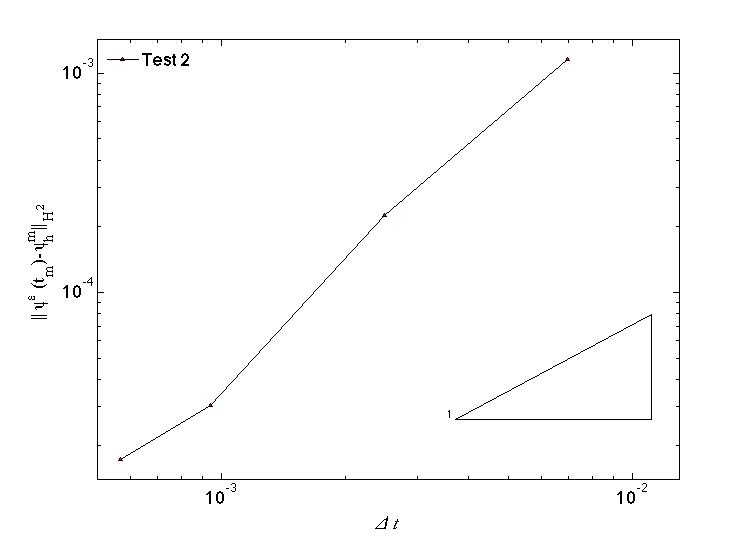}
\includegraphics[angle=0,width=7.25cm,height=6.25cm]{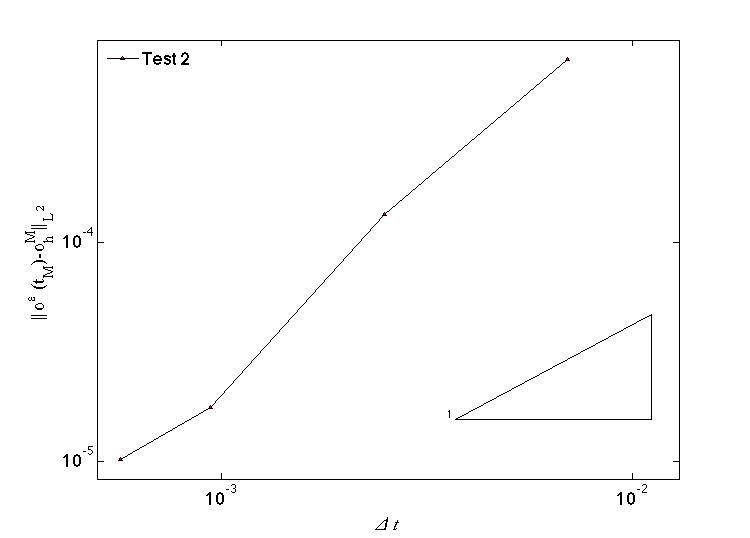}\end{center}
\caption{\label{figtest4}{\scriptsize Test 2:  Change of
$\|\psi^\eps(t_M)-\psi^M_h\|$ w.r.t. $\Delta t=h^2$. $\eps=0.01$,
$t_m=0.25$.}}
\end{figure}

\begin{table}[tbp]
\begin{center}
{\small \noindent\begin{tabular}{|l|l|l|l|l|l|}\hline  $h$
&$\Delta t$ & $\|\psi^\eps(t_m)-\psi^m_h\|_\lt$ &
$\|\psi^\eps(t_m)-\psi^m_h\|_{H^1}$ &
$\|\psi^\eps(t_m)-\psi^m_h\|_{H^2}$ &
$\|\alpha^\eps(t_m)-\alpha^m_h\|_\lt$\\
\hline 0.08333 &0.00694 &0.000214135 &0.000978608
&0.004434963
&0.003456864\\
\hline0.05    &0.0025  &6.15715E-05& 0.000281367& 0.001274611
&0.001009269\\
\hline 0.03066& 0.00094&  1.42185E-05& 6.49825E-05& 0.000294575&
0.000232896\\
\hline 0.02384 &0.00057 &7.13357E-06 &3.25959E-05 &0.000147586&
0.000116731\\
\hline
\end{tabular}}
\end{center} 
\centerline{\scriptsize  Table 1: Change of $\|\psi^\eps(t_M)-\psi^M_h\|$ w.r.t.
$\Delta t=h^2$.  $\eps=0.01$, $t_m=0.25$.} 
\end{table}

\subsection{Test 3}
For this test, we solve problem \eqref{geomethod1}--\eqref{geomethod2} 
with domain $U=(0,6)^2$ and initial condition
\begin{align*}
\alpha_0(x)=\frac18 \chi_{[2,4]\times [2.25,3.75]}(4-x_1)(x_1-2)(3.75-x_2)(x_2-2.25),
\end{align*}
where $\chi_{[2,4]\times [2.25,3.75]}$ denotes the characteristic 
function of the set $[2,4]\times [2.25,3.75]$. We comment that the exact 
solution of this problem is unknown. We plot the computed $\alpha^m_h$ 
and $\psi^m_h$ at times $t_m=0$, $t_m=0.05$, and $t_m=0.1$, and $t_m=0.15$ 
in Figure \ref{figtest5} with parameters $\Delta t=0.001,\ h=0.05$, and $\eps=0.01$.  
As expected, the figure shows that $\alpha_h^m >0$ and $\psi_h^m$ is convex for all $m$.

\begin{figure}[ht]
\begin{center}
\includegraphics[angle=0,width=7.25cm,height=6.25cm]{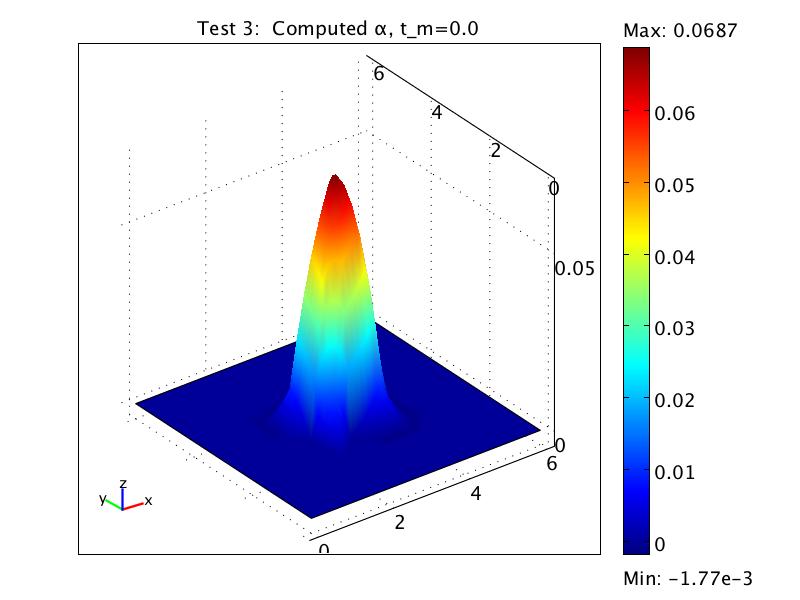}
\includegraphics[angle=0,width=7.25cm,height=6.25cm]{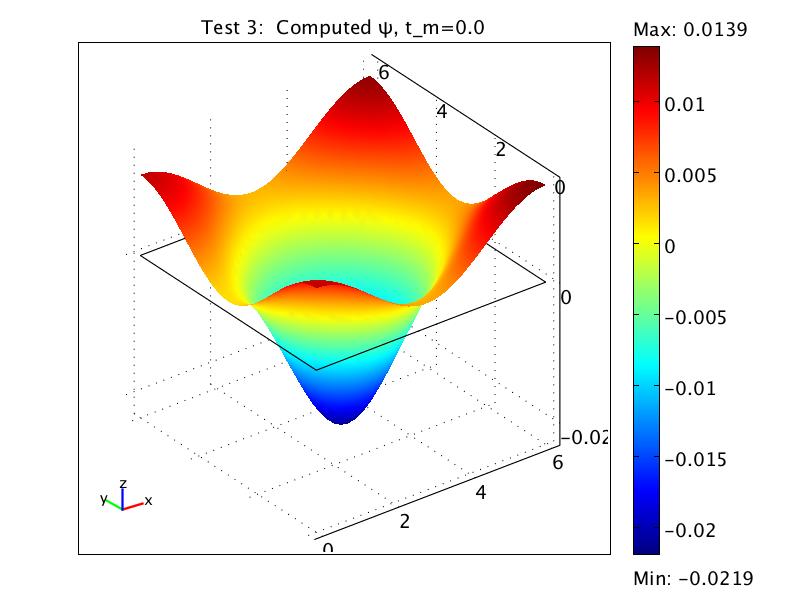}\\
\includegraphics[angle=0,width=7.25cm,height=6.25cm]{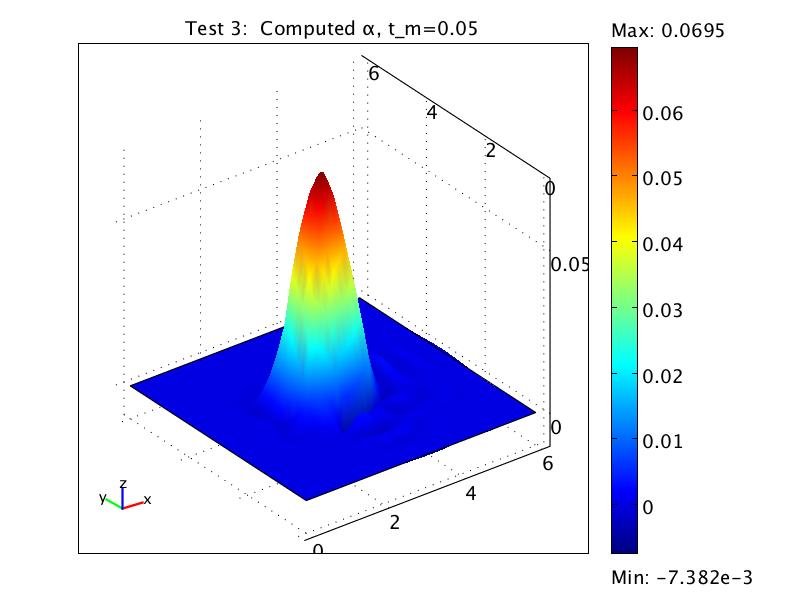}
\includegraphics[angle=0,width=7.25cm,height=6.25cm]{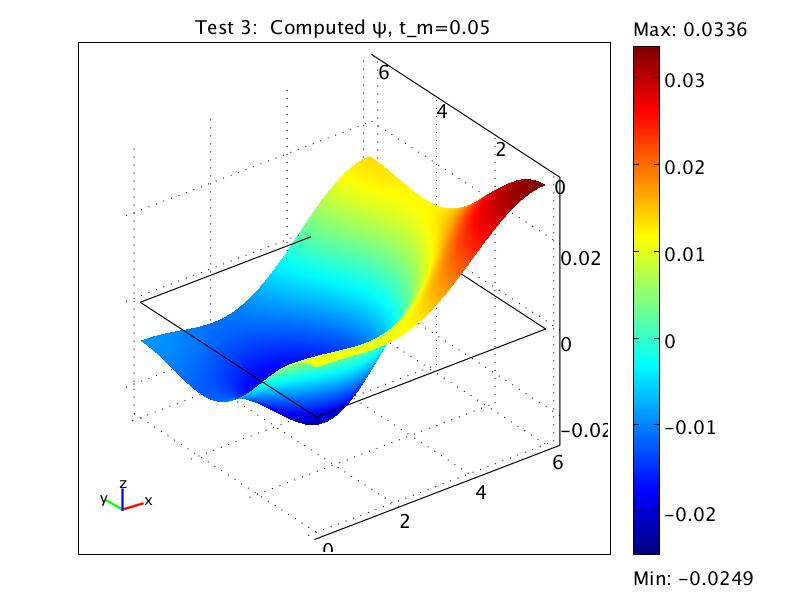}\\
\includegraphics[angle=0,width=7.25cm,height=6.25cm]{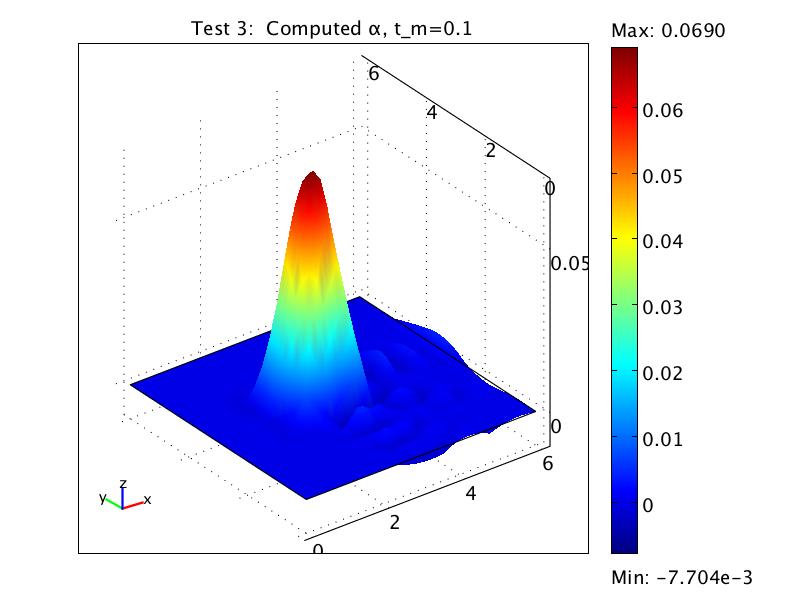}
\includegraphics[angle=0,width=7.25cm,height=6.25cm]{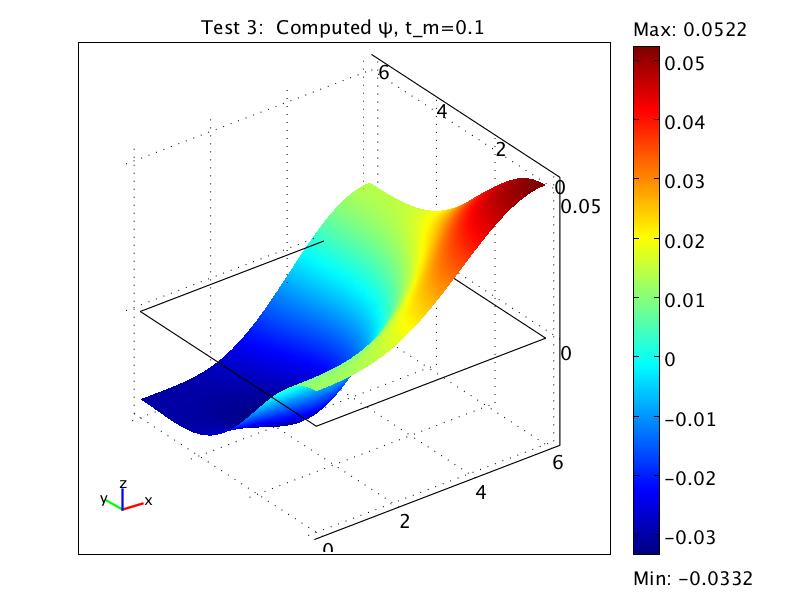}
\end{center}
\caption{\label{figtest5}{\scriptsize Test 3:  Computed $\alpha^m_h$ (left) 
and $\psi^m_h$ (right) at $t_m=0$ (top), $t_m=0.05$ (middle), and $t_m=0.1$ (bottom).
$\Delta t=0.01,\ h=0.05,\ \eps=0.01$}}
\end{figure}



\begin{thebibliography}{99}
\bibliographystyle{abbrv}
\bibitem{Aleksandrov61} A. D. Aleksandrov, {\em Certain estimates for the
Dirichlet problem},  Soviet Math.  Dokl., 1:1151-1154, 1961.

\bibitem{Barles_Souganidis91}
G.~Barles and P.~E. Souganidis, {\em Convergence of approximation
schemes for fully nonlinear second order equations}, Asymptotic
Anal., 4(3):271--283, 1991.

\bibitem{Benamou_98}J. Benamou and Y. Brenier, {\em Weak Existence
for the Semigeostrophic Equations Formulated as a Coupled
Monge-Amp\'ere/Transport Problem}, SIAM. J. Appl. Math.,
58(5):1450-1461, 1998.

\bibitem{Brenner} S. C. Brenner and L. R. Scott, {\em The Mathematical
Theory of Finite Element Methods}, second edition, Springer (2002).

\bibitem{Brenier91}Y. Brenier, {\em Polar factorization and monotone
rearrangement of vector-valued functions.} Comm. Pure Appl. Math.,
44 (1991), pp. 375-417

\bibitem{Caffarelli_Cabre95} L.~A. Caffarelli and X.~Cabr{\'e},
{\em Fully Nonlinear Elliptic Equations}, volume~43 of {\em American
Mathematical Society Colloquium Publications}. American Mathematical
Society, Providence, RI, 1995.

\bibitem{Caffarelli_Milman99} L.~A. Caffarelli and M. Milman, {\em
Monge {A}mp\`ere Equation: Applications to Geometry and
Optimization}, \textsl{Contemporary Mathematics}, Vol. 226, American
Mathematical Society, Providence, RI, 1999.

\bibitem{Ciarlet78} P. G. Ciarlet, {\em The Finite Element Method for Elliptic
Problems.} North-Holland, Amsterdam, 1978.

\bibitem{Crandall_Ishii_Lions92}
M.~G. Crandall, H.~Ishii, and P.-L. Lions, {\em User's guide to
viscosity solutions of second order partial
  differential equations},
Bull. Amer. Math. Soc. (N.S.), 27(1):1--67, 1992.

\bibitem{Cullen_Douglas99} M. Cullen, R. Douglas, {\em Applications
of the Monge-Amp\'ere equation and Monge transport problem to
meteorology and oceanography}, Contemporary Mathematics, 226:33--53, 1999.

\bibitem{cullen_feldman06} M. Cullen and M. Feldman, {\em Lagrangian solutions
of semigeostrophic equations in physical space}, SIAM J. Math. Anal.,
37:1371--1395, 2006.

\bibitem{cnp91} M. Cullen, J. Norbury, and R. J. Purser, {\em Generalized
Lagrangian solutions for atmospheric and oceanic flows}, 
SIAM J. Appl. Math., 51:20--31, 1991.

\bibitem{Dean_Glowinski06b} E.~J. Dean and R.~Glowinski,
{\em Numerical methods for fully nonlinear elliptic equations of the
{M}onge-{A}mp\`ere type}, Comput. Methods Appl. Mech. Engrg.,
195(13-16):1344--1386, 2006.

\bibitem{Douglas_84}J. Douglas, Jr., {\em Numerical Methods for the
Flow of Miscible Fluids in Porous Media} in Numerical  Methods in
Coupled Systems (R. W. Lewis, P. Bettess, and E. Hinton eds.), John
Wiley \& Songs, New York.

\bibitem{Douglas_Russell_82} J. Douglas, Jr. and T. Russell, {\em Numerical 
Methods for Convection-Dominated Diffusion Problems
Based on Combining the Method of Characteristics with Finite
Element or Finite Difference Procedures}, SIAM. J. Numer. Anal.,
19(5):871-885, 1982.

\bibitem{evans} L. C. Evans, {\em Partial Differential Equations},
volume 19 of \textsl{Graduate Studies in Mathematics},  American
Mathematical Society, Providence, RI, 1998.

\bibitem{Feng1} X. Feng, {\em Convergence of the vanishing moment method
for the Monge-Amp\'ere equations in two spatial dimensions}, Trans.
AMS, (submitted)

\bibitem{Feng2} X. Feng and M. Neilan, {\em Vanishing moment method and
moment solutions for second order fully nonlinear partial
differential equations}, DOI 10.1007/s10915-008-9221-9, 2008

\bibitem{Feng3} X. Feng and M. Neilan, {\em Mixed finite element methods
for the fully nonlinear Monge-Amp\'ere equation based on the
vanishing moment method}, SIAM J. Numer. Anal. (submitted)

\bibitem{Feng4}X. Feng and M. Neilan, {\em Analysis of Galerkin
methods for the fully nonlinear Monge-Amp\'ere equation}, Math.
Comp. (submitted)

\bibitem{Gilbarg_Trudinger01} D. Gilbarg and N. S. Trudinger,  \textsl{Elliptic
Partial Differential Equations of Second Order},  \textsl{Classics
in Mathematics},  Springer, Berlin, 2001.  Reprint of the
1998 edition.

\bibitem{hoskins75} B. J. Hoskins, {\em The geostrophic momentum approximation 
and the semigeostrophic equations}, J. Atmospheric Sci., 32:233--242, 1975.

\bibitem{Loeper_06}G. Loeper, {\em A Fully Non-linear Version of the Incompressible
Euler Equations:  The Semi-Geostrophic System},
http://arxiv.org/abs/math/0504138v1.

\bibitem{Majda} A. Majda, {\em Introduction to PDEs and Waves for Atmosphere and 
Ocean}, American Mathematical Society, 2003.

\bibitem{Neilan_08} M. Neilan, {\em A nonconforming Morley finite element method for the fully nonlinear Monge-Amp\'ere equation}, Numer. Math. (submitted)

\bibitem{Neilan} M. Neilan, {\em Numerical Methods for Second Order Fully
Nonlinear PDEs}, Ph.D. Dissertation, the University of Tennessee (in preparation).

\bibitem{McCann_Oberman04}
R.~J. McCann and A.~M. Oberman. {\em Exact semigeostrophic flows in
an elliptical ocean basin}, Nonlinearity, 17(5):1891--1922, 2004.

\bibitem{Oberman07}
A.~M. Oberman, {\em Wide stencil finite difference schemes for
elliptic Monge-Amp\'ere equation and functions of the eigenvalues of
the hessian}, Discret. Cont. Dynam. Sys. B, 10(1):221-238, 2008. 

\bibitem{Oliker_Prussner88}
V.~I. Oliker and L.~D. Prussner, {\em  On the numerical solution of
the equation {$(\partial\sp 2z/\partial
x\sp 2)(\partial\sp 2z/\partial y\sp 2)-((\partial\sp 2z/\partial x\partial
y))\sp 2=f$} and its discretizations. {I}.},
Numer. Math., 54(3):271--293, 1988.

\bibitem{Salmon} R. Salmon, {\em Lectures on geophysical fluid dynamics},
Oxford University Press, New York, 1998.

\end{thebibliography}
\end{document}